\documentclass[a4paper,reqno]{amsart}

 \usepackage{amsmath,amsthm, amssymb}
\usepackage{mathrsfs}
\usepackage[shortlabels]{enumitem}
\usepackage{graphicx}
\usepackage[font=small]{caption}
\usepackage{xcolor}
\usepackage{url}

\newcommand{\N}{\mathbb{N}}
\newcommand{\R}{{\mathbb{R}}}
\newcommand{\C}{{\mathbb{C}}}
\newcommand{\Z}{{\mathbb{Z}}}

\newcommand{\dd}{{{\rm d}}}
\newcommand{\ii}{{\rm i}}
\newcommand{\e}{{\rm e}}

\newcommand{\cf}{\emph{cf.}}
\newcommand{\ie}{{\emph{i.e.}}}
\newcommand{\eg}{{\emph{e.g.}}}

\newcommand{\ov}{\overline}

\newcommand\ds{\displaystyle}

\newcommand{\la}{\lambda}

\newcommand{\eps}{\varepsilon}

\renewcommand{\H}{{\mathcal{H}}}

\newcommand{\spd}{\sigma_{\rm disc}}

\newcommand{\se}[1]{\sigma_{\rm e#1}}
\newcommand{\spp}{\sigma_{\rm p}}

\newcommand{\essinf}{\operatorname*{ess \,inf}}
\newcommand{\esssup}{\operatorname*{ess \,sup}}

\newcommand{\Dom}{{\operatorname{Dom}}}
\newcommand{\Ker}{{\operatorname{Ker}}}
\newcommand{\Ran}{{\operatorname{Ran}}}

\renewcommand{\Re}{\operatorname{Re}}
\renewcommand{\Im}{\operatorname{Im}}
\newcommand{\dist}{\operatorname{dist}}
\newcommand{\sgn}{\operatorname{sgn}}
\newcommand{\supp}{\operatorname{supp}}

\newcommand{\loc}{\mathrm{loc}}
\newcommand{\BigO}{\mathcal{O}}
\newcommand{\os}{{\rm o}}

\newcommand{\Rd}{\mathbb{R}^d}

\newcommand{\LtlocOm}{{L^2_{\loc}(\Omega)}}
\newcommand{\LtlocOmR}{{L^2_{\loc}(\Omega;\R)}}
\newcommand{\LOm}{{L^2(\Omega)}}

\newcommand{\LolocOmR}{{L^1_{\loc}(\Omega;\R)}}
\newcommand{\LiRd}{{L^{\infty}(\Rd)}}

\newcommand{\CcRd}{{C_0^{\infty}(\Rd)}}
\newcommand{\CcOm}{{C_0^{\infty}(\Omega)}}
\newcommand{\WotRd}{{W^{1,2}(\Rd)}}

\newcommand{\WotOmD}{{W^{1,2}_0(\Omega)}}

\newcommand{\Dt}{-\frac{\dd^2}{\dd x^2}}

\newcommand{\DD}{\Delta_{\rm D}}

\newcommand{\wto}{\xrightarrow{w}}
\newcommand{\lto}{\longrightarrow}

\theoremstyle{plain}

\newtheorem{theorem}{Theorem}[section]
\newtheorem{lemma}[theorem]{Lemma}

\newtheorem{proposition}[theorem]{Proposition}

\theoremstyle{definition}

\newtheorem{remark}[theorem]{Remark}
\newtheorem{asm}{Assumption}

\newcommand\cB{\mathcal B}
\newcommand\cC{\mathcal C}
\newcommand\cD{\mathcal D}

\newcommand\cH{\mathcal H}

\newcommand\fra{\mathfrak a}
\newcommand\frq{\mathfrak q}

\usepackage[normalem]{ulem}

\usepackage{mathtools}
\mathtoolsset{showonlyrefs}

\newcommand{\Num}{\operatorname{Num}}

\newcommand{\Cla}{\C \setminus (-\infty,0]}
\newcommand{\ClaO}{\C \setminus (-\infty,0)}
\newcommand{\Claq}{\C \setminus (-\infty,-\alpha_q]}
\newcommand{\ai}{a_{\rm \inf}}
\newcommand{\ar}{a_{\rm r}}

\newcommand{\as}{a_{\rm s}}
\newcommand{\Ms}{M_{\rm s}}
\newcommand{\qr}{q_{\rm r}}

\newcommand{\qs}{q_{\rm s}}

\numberwithin{equation}{section}

\begin{document}


\title{The damped wave equation with unbounded damping}

\author{Pedro Freitas}
\address[Pedro Freitas]{
Departamento de Matem\'{a}tica, Instituto Superior T\'{e}cnico, Universidade de Lisboa, Av. Rovisco Pais, 1049-001 Lisboa, Portugal
\&
Grupo de F\'{\i}sica Matem\'{a}tica, Faculdade de Ci\^{e}ncias, Universidade de Lisboa, Campo Grande, Edif\'{\i}cio C6,
1749-016 Lisboa, Portugal }
\email{psfreitas@fc.ul.pt}
%
%
\author{Petr Siegl}
\address[Petr Siegl]{
Mathematisches Institut, 
Universit\"{a}t Bern,
Alpeneggstr.~22,
3012 Bern, Switzerland
\& On leave from Nuclear Physics Institute CAS, 25068 \v Re\v z, Czechia}
\email{petr.siegl@math.unibe.ch}

\author{Christiane Tretter}
\address[C.\ Tretter]{
	Mathematisches Institut, 
	Universit\"{a}t Bern,
	Sidlerstrasse 5,
	3012 Bern, Switzerland
}
\email{tretter@math.unibe.ch}

\thanks{P.F.\ was partially supported by the Funda\c c\~{a}o para a Ci\^{e}ncia e Tecnologia, Portugal, through project PTDC/MAT-CAL/4334/2014.
The research of P.S.\ is supported by the \emph{Swiss National Science Foundation} Ambizione grant No.~PZ00P2\_154786.
C.T.\ gratefully acknowledges the support of the \emph{Swiss National Science Foundation}, SNF, grant no.\ $169104$.
}



\date{February 20, 2018}

\begin{abstract}
We analyze new phenomena arising in linear damped wave equations on unbounded domains when the damping is allowed to become unbounded at infinity.
We prove the generation of a contraction semigroup, study the relation between the spectra of the semigroup generator and the associated quadratic
operator function, the convergence of non-real eigenvalues in the \mbox{asymptotic} regime of diverging damping on a
subdomain, and we~investigate the appearance of essential spectrum on the negative real axis. We~further show that the presence of the latter prevents
exponential estimates for the semigroup and turns out to be a robust effect that cannot be easily canceled by adding a positive potential.
These analytic results are illustrated by examples.
\end{abstract}

\maketitle

\section{Introduction}

We consider the spectral problem associated with the linearly damped wave equation 
\begin{equation}\label{DWE.me}
\begin{aligned}
u_{tt}(t,x) + 2 a(x) u_t(t,x) = (\Delta - q(x)) u(t,x), \quad t > 0, \quad x \in \Omega ,
\end{aligned}
\end{equation}
with non-negative damping $a$ and potential $q$ on an open (typically unbounded) subset $\Omega$ of $\R^d$; when $\Omega$ is not all of $\R^d$ 
we shall impose Dirichlet boundary conditions on its boundary $\partial \Omega$. Here both the potential $q$
and the damping $a$ are allowed to be unbounded and/or singular. 

The main goal of this paper is to analyze the new phenomena which arise when the damping term $a$ is allowed to grow to infinity on an unbounded domain. To this end, we formally rewrite~\eqref{DWE.me} as a first order~system 
\begin{equation}\label{DWE.system}
\partial_t 
\begin{pmatrix}
u \\
v
\end{pmatrix}
=
\underbrace{
\begin{pmatrix}
0 & I \\
\Delta - q & - 2 a  
\end{pmatrix}}_{G}
\begin{pmatrix}
u \\
v
\end{pmatrix}
\end{equation}
and realize the operator $G$ in a suitable Hilbert space without assuming that the damping is dominated by $\Delta-q$. Our main results show that, even under these weak assumptions on the damping, $G$ generates a contraction semi-group, but $G$ may have essential spectrum 
that covers the entire semi-axis $(-\infty,0]$. As a consequence, 
although the energy of solutions will still approach zero, this decay will now be polynomial and
no longer exponential, \cf~\cite{Ikehata-2017} and the discussion
below. We further establish conditions for the latter and study 
the convergence of non-real eigenvalues in the asymptotic regime of diverging damping on a subdomain. 

In most of the literature on linearly damped wave equations on unbounded domains only bounded damping terms were considered. This is a natural condition to allow for the exponential decay of solutions, while large damping terms in fact tend to weaken the decay giving rise to the phenomenon known as over-damping. More precisely, increasing the damping term past a certain threshold will cause part of the spectrum to approach the imaginary axis, thus producing a slower decay. This phenomenon may already occur in finite-dimensional systems and in equations
like~\eqref{DWE.me} with bounded damping where its effect on individual eigenvalues is well-understood. Unbounded accretive or sectorial damping terms of equal strength as $-\Delta$ were considered as an application of semigroup generation results and of spectral estimates for second order abstract Cauchy problems in \cite{Jacob-2009-79,Jacob-2017} which allow to control the spectrum, in particular, near the imaginary axis.

To the best of our knowledge, the only article where the damping has been allowed to become unbounded at infinity is the recent preprint~\cite{Ikehata-2017}.
There the authors consider dampings on all of $\Rd \;(d\geq 3)$ and, using methods different from ours, they prove the existence and uniqueness of weak solutions whose energy decays at least with $(1+t)^{-2}$. In fact, our result on the essential spectrum will show that one of the characteristics of such systems is that the essential spectrum covers the whole semi-axis $(-\infty,0]$, thus excluding exponential energy decay.

To illustrate this issue, consider the simple model case given by the generators $G_n$, $n\in \N$, of the wave equation on the real line with the family of damping terms 
\begin{equation}\label{an.ex}
a_n(x)=x^{2n}+\fra_0, \quad x \in \Omega =\R, \quad n \in \N, \quad \fra_0 \geq 0,
\end{equation}
and a constant potential $q(x)=\frq_0 \geq 0$, $x \in \R$. The formal limit as $n$ goes to $\infty$ leads to a simple problem (in a different space) with
\begin{equation}\label{ainf.ex}
a_\infty(x)= \fra_0, \quad x \in \Omega_\infty=(-1,1),
\end{equation}
and Dirichlet boundary conditions at $\pm 1$. The spectrum of the generator $G_\infty$ is discrete and may be found explicitly.  It consists of eigenvalues with all, but possibly finitely many, lying on the line $-\fra_0 + \ii \R$, \cf~Remark~\ref{rem:x.inf}. Moreover, the energy of solutions of the corresponding wave equation is known to decay exponentially, \cf~for instance~\cite{Freitas-1996-132}. A natural question is to what extent the properties of $G_\infty$ are shared by $G_n$. The non-real eigenvalues of $G_n$, here given explicitly in terms of the eigenvalues of $-\dd^2/\dd x^2 + x^{2n} $
and located on rays of the form $\e^{\pm \ii \frac{n+1}{2n+1}\pi} \R_+$, \cf~Proposition~\ref{prop:x2n.ex}, do indeed converge to those of $G_\infty$.
However, while the spectrum of $G_\infty$ is discrete and does not contain $0$, all $G_n$, $n\in \N$, have non-empty essential
spectrum covering the entire negative semi-axis and thus $0$ is in the spectrum of $G_n$.
As a consequence, exponential decay of energy is lost, \cf~for instance~\cite[Thm.~10.1.7]{Davies-2007}.

The fundamental point here is that the essential spectrum can no longer be shifted away from $0$ by adding a positive potential $q(x)\geq \frq_0 >0$, as
might be done for bounded damping, to ensure that exponential energy decay still holds, \cf~for instance \cite{Zuazua-1991-70,Nakao-2001-238}. In fact,
even a potential $q$ that is unbounded at infinity, but does not dominate the damping term, will not be enough to cancel this effect.  On the other hand,
a dominating potential $q$ can be used to shift the essential spectrum from $0$, \cf~Remark~\ref{rem:q.dom}. 

As we will see, \eqref{an.ex} is not an isolated example and our results cover the much more general setting  with an open (typically unbounded) domain $\Omega \subset \Rd$,
a potential $q$ with low regularity and a damping
$a$ satisfying natural conditions allowing for a convenient separation property of the domain of the Schr\"odinger operator $-\Delta +q + \gamma a$
when $\gamma$ belongs to $\Cla$, \cf~Assumption~\ref{asm:a.r}, Remark~\ref{rem:sep}.i) and Theorem~\ref{thm:dom.H.gamma}.

We emphasize that the unbounded damping at infinity can by no means be viewed as ``small'' when compared to $-\Delta + q$ and our results, even those which
are qualitative, do not follow by standard perturbation techniques, traditionally used to handle bounded or relatively bounded damping terms.

The proofs rely on a wider range of methods like elliptic estimates for Schr\"odinger operators with unbounded complex potentials, quadratic complements (quadratic operator functions associated with \eqref{DWE.system}),  Fredholm theory, the use of suitable notions of essential spectra for non-self-adjoint operators, WKB expansions, convergence of sectorial forms acting in different spaces with $L^1_{\rm loc}$-coefficients, spectral convergence of holomorphic operator families, and properties of solutions of second order ODE's with polynomial potentials.

The crucial part of our analysis is the relation between the spectrum, and some of its subsets, of the generator $G$ and the associated quadratic operator function $T$ given by
\begin{equation}\label{T.act.def}
T(\la) = -\Delta + q + 2\la a + \la^2, \quad \la \in \Cla.
\end{equation}
While for bounded damping $T(\la)$ is defined for all $\la \in \C$ and the equivalence of $\la \in \sigma(G)$ and $0\in \sigma(T(\la))$ is relatively straightforward, \cf~for instance~\cite[Sec.~2.2, 2.3]{Tretter-2008} for abstract results, 
the unboundedness of $a$ is a major challenge that requires a new approach; in particular, first $T(\la)$ has to be introduced as a closed operator with non-empty resolvent set acting in $\LOm$.

It is the precise description of $\Dom(T(\la))$, \cf~Theorem~\ref{thm:dom.H.gamma}, that enables us to prove both the generation of a contraction semigroup,
\cf~Theorem~\ref{thm:m-ac}, and the spectral correspondence between $G$ and $T(\la)$ for the restricted range $\la \in \Cla$, \cf~Theorem~\ref{thm:GT}. Clearly,
there are crucial differences between $-\Delta + 2\la a$ for $\la >0$ and $\la<0$ since the quadratic form of the latter is not semi-bounded. Nevertheless,
for a general $\la \in \C$, convenient properties of $T(\lambda)$ with $\lambda >0$ remain valid also for $\la \in \Cla$ since the possibly negative
real part of $2 \la a$ is compensated by the imaginary part of $2 \la a$, \cf~Section~\ref{sec:T.def} for details. 

When $a$ is unbounded at infinity, 
we show that the set $\{\la\in \Cla:0 \in \sigma(T(\la))\}$, and hence the non-real spectrum of $G$, consists only of discrete eigenvalues of finite multiplicity which may only accumulate at the
semi-axis $(-\infty,0]$. Since the unboundedness of $a$ is not required for the equivalences in Theorem~\ref{thm:GT}, also the non-real essential spectrum of $G$ can be analyzed by studying whether $0$ belongs to the essential spectrum of $T(\la)$.

Because $T(\la)$ is not defined for $\la\in (-\infty,0)$, the negative real spectrum of $G$ is investigated directly for unbounded domains $\Omega$. We show that if $a$ grows to infinity in a channel in $\Omega$ whose radius may shrink at $\infty$ at a rate controlled by the growth of $a$, then $0$ belongs to the essential
spectrum of $G$. In fact, the whole real negative semi-axis belongs to the essential spectrum of $G$ even when $q$ is unbounded but does not
dominate $a$, \cf~Theorem~\ref{thm:Gspe}.

In Section~\ref{sec:lim.an}, motivated by examples \eqref{an.ex}, \eqref{ainf.ex} above, we prove a convergence result for non-real eigenvalues and corresponding eigenfunctions of a sequence of quadratic functions $\{T_n(\la)\}_n$, $\la \in \Cla$, with dampings possibly diverging on a subset of $\Omega$, \cf~Theorem~\ref{thm:sp.conv}. We thus justify the formal limit considered in the examples above.

In Section~\ref{sec:ex} we analyze two examples, the first is on the whole real line based on~\eqref{an.ex} and~\eqref{ainf.ex}, while the second is on a horizontal strip in $\R^2$ with damping $a(x,y) = x^2 + \fra_0$ and the corresponding discrete spectrum displaying the structure of a two--dimensional problem. Apart from showing what type of behavior one may now expect from isolated eigenvalues, more importantly both cases illustrate that having the discrete spectrum to the left of a line $\Re{\lambda}=-\alpha_{0}<0$ is, by itself, not enough to determine the type of decay of solutions in the presence of unbounded damping. Indeed, our results applied to both examples show that the essential spectrum covers the negative part of the real axis all the way up to $0$, thus excluding the possibility of uniform exponential decay of solutions in general.

\subsection{Notation}
\label{sec:not}

The following notations and conventions are used throughout the paper. The norm and inner product (linear in the first entry) in $\LOm$ are denoted by 
$\|\cdot\|:=\|\cdot\|_\LOm$ and $\langle \cdot, \cdot \rangle := \langle \cdot, \cdot \rangle_\LOm$, respectively. The domain of a  multiplication operator by a measurable function $m$ (here $a$ and $q$) in $L^2(\Omega)$ is always taken to be maximal, \ie~ 
\begin{equation}
\Dom(m):= \{ \psi \in \LOm \, : \, m \psi \in \LOm \}.
\end{equation}
The Dirichlet Laplacian on $\Omega$, introduced through the corresponding form, is denoted by $\DD$, \ie~
\begin{equation}
\begin{aligned}
- \DD \psi = -\Delta \psi,\quad \Dom(\DD) = \{ \psi \in \WotOmD \; : \: \Delta \psi \in \LOm\}.
\end{aligned}
\end{equation}
When $-\Delta+q$ is viewed as an operator, the Dirichlet realization introduced through the form is meant, \ie~
\begin{equation}\label{Dom.Dq.def}
\Dom(-\Delta +q):= \{\phi \in \WotOmD \cap \Dom(q^\frac12) \, : \, (-\Delta +q)\phi \in \LOm \}. 
\end{equation}
For $\Omega_1 \subset \Omega$ we view $L^2(\Omega_1)$ as a subspace of $\LOm$, $\LOm = L^2(\Omega_1) \oplus L^2(\Omega \setminus \ov \Omega_1)$,
\ie~we use zero extensions. On the other hand, for $f \in \LOm$, $\|f\|_{L^2(\Omega_1)}$ means $\|f \restriction \Omega_1 \|_{L^2(\Omega_1)}$.
For consistency with earlier work, we denote the numerical range of a linear operator $A$ acting in a Hilbert space $\cH$ by
\begin{equation}
\Num(A):= \{ \langle A f, f \rangle_\cH \, : \, f \in \Dom(A), \ \|f\|_\cH =1 \},
\end{equation}
while $W(A)$ may be more common in the operator theoretic literature; the numerical range of a quadratic form is introduced analogously, \cf~\cite[Sec.~VI]{Kato-1966}.

The essential spectrum of a non-self-adjoint operator may be defined in several, different and in general not equivalent, ways. Here we
use the definition via Weyl singular sequences, denoted by $\se{2}(\cdot)$ in \cite[Sec.~IX]{EE},
\begin{equation}
\se{2}(A) = \{ \la \in \C: \ \exists \{\psi_n\} \subset \Dom(A), \ \|\psi_n\|=1, \ \psi_n \wto 0, (A-\la) \psi_n \to 0, \ n \to \infty\}.
\end{equation}

\section{Generation of a contraction semigroup}
\label{sec:G.def}
Throughout the paper, if not stated otherwise, we shall assume that the damping and the potential satisfy the following regularity conditions.
\begin{asm}[\emph{Regularity assumptions on the damping $a$ and the potential $q$}]
\label{asm:a.r}
Let $a \in \LtlocOmR$, $q \in \LolocOmR$ with $a,q \geq 0$. Suppose that $a$ can be decomposed into a regular and singular part as
\begin{equation}
a=\ar + \as, \quad \ar \geq 0, 
\end{equation}
with $\ar \in W^{1,\infty}_{\rm loc}(\ov \Omega;\R)$, $\as \in \LtlocOmR$ and, for every $\eps>0$, there exists a constant $M_\nabla=M_\nabla(\eps) > 0$ such that
\begin{equation}\label{asm:a.reg}
|\nabla \ar| \leq \eps \ar^{\frac 32} + M_\nabla(q^\frac12 +1).
\end{equation}
Further assume that, for every $\eps >0$, there exists a constant $\Ms=\Ms(\eps)>0$ such that, for all $\psi \in \Dom(\ar) \cap \Dom(-\Delta +q)$,
\cf~\eqref{Dom.Dq.def},
\begin{equation}\label{asm:a2}
\begin{aligned}
\|\as \psi\| \leq \eps (\|(-\Delta+q) \psi\|+\|\ar \psi\|) + \Ms \|\psi\|. 
\end{aligned}
\end{equation}
\end{asm}
\begin{remark}
\label{rem:sep}
The exponent $\frac 32$ in~\eqref{asm:a.reg} is known to be optimal for the so-called separation property of the domain of
$-\Delta + a$, \cf~for instance~\cite{Everitt-1978-79}, which will be proved (and used) here as well, \cf~Theorem~\ref{thm:dom.H.gamma}.
\end{remark}
In some cases, we will assume, in addition, that $a$ is unbounded at infinity which results in special spectral features like
in Proposition~\ref{prop:sp.T} or Theorem~\ref{thm:GT}. 
\begin{asm}[\emph{Unboundedness of damping $a$ at infinity}]
\label{asm:a.gr} 
Let	$a$ satisfy 
\begin{equation}\label{asm:a.inf}
\lim_{k \to \infty} \  \essinf_{x \in \Omega, |x|>k} \ a(x) = \infty.
\end{equation}
\end{asm}

We are mostly interested in the situation when $a$ is not dominated by $q$, and a typical potential $q$ being bounded (or even $0$). The case where $q$ dominates $a$ is discussed in Remark~\ref{rem:q.dom}. 

In order to find a suitable operator realization of the formal operator matrix $G$, \cf~\eqref{DWE.system}, 
we denote by $\mathcal W(\Omega)$ the completion of the
pre-Hilbert space 
\begin{equation}
\left(\CcOm,
\langle \nabla \cdot, \nabla \cdot \rangle +  \langle q^\frac 12 \cdot,
q^\frac 12 \cdot \rangle
\right)
\end{equation}
the inner product of which is non-degenerate since $\nabla$ is injective on $\CcOm$,
and we define the product Hilbert space 
\begin{equation}\label{Hilb.space}
\begin{aligned}
\cH &:= 
\mathcal W(\Omega) \times \LOm,
\\
\langle (\phi_1, \phi_2), (\psi_1, \psi_2) \rangle_\H &:= \langle \nabla \phi_1, \nabla \psi_1 \rangle +  \langle q^\frac 12 \phi_1,
q^\frac 12 \psi_1 \rangle + \langle \phi_2, \psi_2 \rangle.
\end{aligned}
\end{equation}
Here $\Dom((-\Delta+q)^{1/2})=\WotOmD \cap \Dom(q^{1/2}) \subset \mathcal W(\Omega)$ and equality holds if, for example, there is a positive constant $\frq_0$ such that $q \geq \frq_0>0$, 
\cf~\cite[Thm.~1.8.1]{Davies-1989}, or if $\Omega$ has finite width and so Poincar\'e's inequality applies, \cf~for instance~\cite[Thm.~6.30]{Adams-2003}; 
then $-\Delta+q$ is uniformly positive and the space $\cH$ in \eqref{Hilb.space} coincides with the usual choice of space for abstract operator matrices associated with quadratic operator functions 
in this case, \cf~for instance \cite{Langer-2006-267}, \cite{Jacob-2017}.

Moreover, by the first representation theorem~\cite[Thm.~VI.2.1]{Kato-1966}, $\Dom(-\Delta+q)$ and also its core $\cD$ given by the restriction to functions with compact support, \cf~\eqref{core.def}, are dense in $\mathcal W(\Omega)$.

In $\cH$ we introduce the densely defined operator
\begin{equation}\label{G0.def}
G_0:=
\begin{pmatrix}
0 & I \\
\Delta - q & - 2 a  
\end{pmatrix},
\qquad 
\Dom(G_0):= \left(\Dom(-\Delta+q) \cap \Dom(a) \right)^2.
\end{equation}

The following theorem states the fundamental property that 
\begin{equation}\label{G.def}
G:=\ov{G_0}
\end{equation}
generates a contraction semigroup; the proof is given at the end of Section~\ref{sec:T.def} after all necessary ingredients have been derived.
\begin{theorem}\label{thm:m-ac}
	Let $a$, $q$ satisfy Assumption~\ref{asm:a.r} and let $G_0$ be as in \eqref{G0.def}. Then $-G_0$ is accretive and $\Ran(G_0 - 1)$
	is dense in $\H$. Hence $-G_0$ is closable with m-accretive closure $-G=-\ov{G_0}$ and $G$ generates a contraction semigroup in $\H$.
\end{theorem}

\subsection{The associated quadratic operator function}
\label{sec:T.def}

Employing sectorial forms, we introduce the family $T(\lambda)$, $\la \in \Cla$, \cf~\eqref{T.act.def}, of closed operators in $\LOm$. Although the operator function $T$ resembles one of the quadratic complements of $G$, \cf~\cite[Sec.~2.2]{Tretter-2008}, however,
here $T(\la)$ is considered as an operator from $\LOm$ to $\LOm$ and not from $\mathcal W(\Omega)$ to $\LOm$. 

We shall introduce $T(\lambda)$ as
\begin{equation}\label{T.def}
T(\lambda):= H_{2\lambda} + \lambda^2, 
\quad 
\Dom(T(\lambda)) 
:= 
\Dom(H_{2\lambda}),  \quad \la \in \Cla,
\end{equation}
via the one-parameter family of operators
\begin{equation}\label{H.gam.def}
H_{\gamma} = -\Delta + q + \gamma a, \quad \gamma \in \ClaO,
\end{equation}
which will be defined rigorously below, using the first representation theorem for a {\it rotated} version of $H_{\gamma}$. In fact, the numerical range of $H_{\gamma}$ 
is contained in a sector with semi-angle smaller then $\pi/2$ which need not lie in the right half-plane; however, after multiplication of $H_{\gamma}$ by~$\e^{-\ii \arg(\gamma) /2}$, we obtain a sectorial operator $\widetilde H_\gamma=\e^{-\ii \arg(\gamma) /2} H_{\gamma}$. We mention that the operator family $\widetilde H_\gamma$, $\gamma \in \Cla$, is not uniformly sectorial, \cf~\eqref{h.gam.nr} below.

Note that here we have included $\gamma=0$ on purpose, although the domains of~$H_\gamma$ for $\gamma\ne 0$ and for $\gamma=0$ are very different. Clearly, for $\gamma = 0$, no rotation is needed since $H_0$ is self-adjoint and bounded from below. For convenience, we set $\arg(0)=0$ in what follows.

In the definition of $H_\gamma$ as well as in several auxiliary results, it suffices to assume less regularity of $a$ than in Assumption~\ref{asm:a.r}. 

\begin{lemma}\label{lem:h.def}
	Let $a,q \in \LolocOmR$ and $a,q \geq 0$. Then the following hold.
	\begin{enumerate}[{\upshape i)}]
		\item For fixed $\gamma \in \ClaO$, the form 
		\begin{equation}
		\begin{aligned}
		\widetilde{h}_{\gamma}  
		&:= \e^{-\frac \ii 2  \arg(\gamma) }  (\| \nabla \cdot \|^2 + \| q^\frac 12 \cdot \|^2)
		+
		\e^{\frac \ii 2  \arg(\gamma) } \,  |\gamma| 
		\|a^\frac 12 \cdot\|^2,
		\\
		\Dom(\widetilde{h}_{\gamma}) 
		&:= 
		\WotOmD \cap \Dom(\gamma a^\frac 12) \cap \Dom(q^\frac 12), 
		\end{aligned}
		\end{equation}
		is closed in $\LOm$ and sectorial with 
		\begin{equation}\label{h.gam.nr}
		\Num (\widetilde h_\gamma) \subset \left\{  z \in \C \, : \, |\arg z| \leq \frac{\arg(\gamma)}2 \right\}.
		\end{equation} 
		\item \label{H.core}
		$\CcOm$ is a core of $\,\widetilde h_\gamma$ and $\,\widetilde h_\gamma$ determines a unique m-sectorial operator~$\widetilde H_\gamma$ in $\LOm$. 
		\item \label{H.comp.res}
		If Assumption \ref{asm:a.gr} holds, then $\widetilde H_{\gamma}$, $\gamma \in \Cla$, has compact resolvent.
		\item \label{H.hol.it}
		The operator family
		\begin{equation}\label{H.gam.def.2}
		H_{\gamma} := \e^{\frac \ii 2  \arg(\gamma) } \widetilde{H}_{\gamma}, \quad \gamma \in \Cla,
		\end{equation}
		is a holomorphic family of closed operators.
	\end{enumerate}
\end{lemma}
\begin{proof}
	i) We denote $\omega:=\arg(\gamma)/2 \in (-\pi/2,\pi/2)$. Since
	\begin{equation}\label{reim.gam}
	\begin{aligned}
	\Re \widetilde{h}_{\gamma} [\psi] 
	&= \cos \omega 
	\left( \|\nabla \psi\|^2 + \|q^\frac 12 \psi \|^2
	+
	|\gamma| \| a^\frac 12 \psi \|^2
	\right),
	\\
	\Im \widetilde{h}_{\gamma} [\psi] 
	&= -\sin \omega 
	\left( \|\nabla \psi\|^2 + \|q^\frac 12 \psi \|^2
	-
	|\gamma| \| a^\frac 12 \psi \|^2
	\right),
	\end{aligned}
	\end{equation}
	the sectoriality of $\widetilde h_\gamma$ follows from
	\begin{equation}\label{tan.gam}
	\begin{aligned}
	\left| \Im \widetilde{h}_{\gamma} [\psi] \right|
	\leq 
	\left| \sin \omega \right| 
	(\|\nabla \psi\|^2 + \|q^\frac 12 \psi \|^2 + |\gamma| \| a^\frac 12 \psi \|^2)  
	\leq 
	\left|\tan \omega \right| \Re \widetilde{h}_{\gamma} [\psi]. 
	\end{aligned}
	\end{equation}
	The form $\widetilde{h}_{\gamma}$ is closed since $\Re \widetilde{h}_{\gamma}$ is closed, \cf~\cite[Thm.~1.8.1]{Davies-1989}. 
	The enclosure \eqref{h.gam.nr} of the numerical range of $\widetilde h_\gamma$ follows from \eqref{tan.gam}.
	
	ii)
	The core property of $\CcOm$ follows from \cite[Thm.~1.8.1]{Davies-1989} and \cite[Thm.~VI.1.21]{Kato-1966}. The operator $\widetilde H_\gamma$ is determined by the first representation theorem \cite[Thm.~VI.2.1]{Kato-1966}. 
	
	iii) The resolvent of $\widetilde{H}_{\gamma}$ is compact if and only if the resolvent of $\Re \widetilde{H}_{\gamma}$ is 
	compact, \cf~\cite[Thm.~VI.3.3.]{Kato-1966}.  The operator $\Re \widetilde{H}_{\gamma}$, induced by the form
	$\Re \widetilde{h}_{\gamma}$, is self-adjoint and has compact resolvent if $\Dom(\widetilde h_\gamma)$ is compactly embedded
	in $\LOm$, \cf~\cite[Thm.~XIII.67]{Reed4}. If $\Omega$ is bounded, then $W_0^{1,2}(\Omega)$, and hence $\Dom(\widetilde h_\gamma)$,
	is compactly embedded in $\LOm$ by the Rellich-Kondrachov Theorem, \cf~\cite[Thm.~6.3]{Adams-2003}.
	For unbounded $\Omega$, let $\phi \in \Dom(\widetilde h_\gamma)$. The zero extension $\widetilde \phi$ of $\phi$ belongs to
	$\WotRd$, \cf~\cite[Lem.~3.27]{Adams-2003}, and to $\Dom(a_{\rm ext}^{1/2})$ where
	\begin{equation}
	a_{\rm ext}(x):= 
	\begin{cases}
	\hspace{6mm} a(x), & x \in \Omega, 
	\\[1mm]
	\ds \essinf_{x\in \Omega, |x|>k} a(x), & x \notin \Omega, \ |x|=k.
	\end{cases}
	\end{equation}
	Moreover, there exists a non-negative constant $C$, independent of $\phi$, such that
	\begin{equation}
	\|\widetilde \phi\|_{W^{1,2}(\R^d)}^2 + \|a_{\rm ext}^\frac 12 \widetilde \phi\|_{L^2(\R^d)}^2 \leq C (\Re \widetilde h_\gamma [ \phi] + \| \phi\|_{\LOm}^2).
	\end{equation}
	The function $a_{\rm ext}$ satisfies Assumption \ref{asm:a.gr} on $\Rd$ and thus, by Rellich's criterion, \cf~\cite[Thm.~XIII.65]{Reed4},
	$\Dom(\widetilde h_\gamma)$ is compactly embedded in $\LOm$  also for unbounded $\Omega$.

	iv) We verify that $H_\gamma$ is holomorphic (in the sense of \cite[Sec.~VI.1.2]{Kato-1966}) in a neighborhood of any $\gamma_0 \in \C \setminus (-\infty,0]$. The strategy is to use the analyticity of the associated quadratic form. Nonetheless, we first note that, in a neighborhood of~$\gamma_0$, $H_\gamma$ is equal to the operator 
	\begin{equation}
	\e^{\frac{\ii}{2} \arg(\gamma_0) } \widehat H_{\gamma} :=
	\e^{\frac{\ii}{2} \arg(\gamma_0) } (\e^{-\frac{\ii}{2} \arg(\gamma_0)}(- \Delta +q) + \e^{ -\frac{\ii}{2} \arg(\gamma_0)} \gamma a) 
	\end{equation}
	where $\widehat H_{\gamma}$ is the m-sectorial operator introduced through the sectorial form 
	\begin{equation}
	\begin{aligned}
	\widehat{h}_{\gamma}  
	&:= \e^{-\frac{\ii}{2} \arg(\gamma_0)} (\| \nabla \cdot \|^2 + \|q^\frac 12 \cdot \|^2)
	+
	\e^{-\frac{\ii}{2} \arg(\gamma_0) } \gamma
	\|a^\frac 12 \cdot\|^2,
	\\
	\Dom(\widehat{h}_{\gamma}) 
	&:= 
	\WotOmD \cap \Dom(a^\frac 12) \cap \Dom(q^\frac 12) ;
	\end{aligned}
	\end{equation}
	the sectoriality and closedness of $\widehat h_\gamma$ can be verified as in the proof of i), \cf~\eqref{reim.gam}--\eqref{tan.gam}, 
	and the equality of the operators $H_\gamma$ and $\e^{\frac{\ii}{2} \arg(\gamma_0) } \widehat H_{\gamma}$ follows from~\cite[Cor.~VI.2.4]{Kato-1966}.
	Since the rotation in $\widehat{h}_{\gamma} $ is independent of $\gamma$, the form associated with $\widehat H_\gamma$ is obviously 	an analytic family of type (a) in a neighborhood of $\gamma_0$, \cf~\cite[Sec.~VII.4.2]{Kato-1966}. Thus $\widehat H_\gamma$, and hence
	$H_\gamma$, are holomorphic in a neighborhood of~$\gamma_0$.
\end{proof}

In case of higher regularity of $a$ as required in Assumption~\ref{asm:a.r}, we obtain the following separation property of $\Dom(T(\la))=\Dom(H_{2\la})$ which ensures that $T(\la)$ is defined as a sum of unbounded operators. The strategy of the proof is similar to~\cite{Krejcirik-2017-221}, but the different type of potentials used here requires new estimates.

\begin{theorem}\label{thm:dom.H.gamma}
	Let $a,q$ satisfy Assumption \ref{asm:a.r} and let $\Dom(-\Delta+q)$ be as in \eqref{Dom.Dq.def}. 
	Then $T(\la)$,~$\la \in \Cla$, is a holomorphic family of type $(A)$ with
	\begin{equation}\label{T.sep}
	\Dom(T(\la)) = \Dom(-\Delta+q) \cap \Dom(a), \quad \la \in \Cla,
	\end{equation}	
	the set
	\begin{equation}\label{core.def}
	\cD := \{ \psi \in \Dom(-\Delta+q) \, : \, \supp \psi \, \text{ is compact in } \Rd\} \subset \Dom(\ar)
	\end{equation}
	is a core of $T(\la)$, $\la \in \Cla$, and 
	\begin{equation}\label{T.Jsa}
	T(\la)^* = \cC T(\la) \cC = T(\ov \la), \quad \la \in \Cla,
	\end{equation}
	where $\cC$ is the $($antilinear$)$ operator of complex conjugation in $\LOm$. 
\end{theorem}
\begin{proof}
	By \eqref{T.def}, it suffices to analyze $H_\gamma$ with $ \gamma =2\la \in \Cla$.	
	It follows from the first representation theorem, \cf~\cite[Thm.~VI.2.1]{Kato-1966}, that 
	\begin{equation}
	\Dom(H_\gamma) 
	\subset 
	\{
	\psi \in \Dom(\widetilde h_\gamma) \, : \, (-\Delta + q + \gamma a )\psi \in \LOm
	\}.
	\end{equation}
	Similarly, for $H_\gamma^{\rm r}:=-\Delta + q + \gamma \ar$, introduced in the same way as $H_\gamma$ through the form $\widetilde h_\gamma^{\rm r}$, we have
	\begin{equation}
	\Dom(H_\gamma^{\rm r}) 
	\subset 
	\{
	\psi \in \Dom(\widetilde h_\gamma^{\rm r}) \, : \, (-\Delta + q + \gamma \ar )\psi \in \LOm
	\}.
	\end{equation}
	Below we prove that $\cD$ is a core of $H_\gamma^{\rm r}$ and that there exist positive constants $k_1$ and $k_2$ such that,
	for all $\psi \in \cD$,
	\begin{equation}\label{norm.eq.1}
	\begin{aligned}
	&k_1 \left(
	\|(-\Delta +q)\psi \|^2 +  \| \ar \psi\|^2 + \|\psi\|^2
	\right)
	\\
	&
	\leq 
	\|H_\gamma^{\rm r} \psi \|^2 + \|\psi\|^2 
	\leq 
	k_2 \left(
	\|(-\Delta +q)\psi \|^2 +  \|\ar \psi\|^2 + \|\psi\|^2
	\right),
	\end{aligned}
	\end{equation}
	from which it follows that $\Dom(H_\gamma^{\rm r}) = \Dom(-\Delta+q) \cap \Dom(\ar)$. 
	
	By \eqref{asm:a2} in Assumption~\ref{asm:a.r} and \eqref{norm.eq.1}, $\gamma \as$ is a relatively bounded perturbation of
	$H_\gamma^{\rm r}$ with relative bound $0$, thus $\Dom(H_\gamma^{\rm r} + \gamma \as) = \Dom(H_\gamma^{\rm r})=\Dom(-\Delta+q) \cap \Dom(a)$.
	Moreover, $H_\gamma^{\rm r} + \gamma \as \subset H_\gamma$ and a standard perturbation argument shows that, for  sufficiently large
	positive $z$, we have $-\e^{ \ii \arg(\gamma)/2} z \in \rho(H_\gamma^{\rm r} + \gamma \as)$. Hence
	$\e^{- \ii  \arg(\gamma)/2}(H_\gamma^{\rm r} +\gamma \as)$ is m-sectorial,  and $H_\gamma^{\rm r} + \gamma \as = H_\gamma$,
	\cf~\cite[Sec.~V.3]{Kato-1966}.  
	
	To prove \eqref{T.sep}, it therefore remains to be shown that $\cD$ is a core of $H_\gamma^{\rm r}$ and that \eqref{norm.eq.1} holds.
	Take $\psi \in \Dom(H_\gamma^{\rm r})$ and notice that  $\ar \psi \in  L^2_{\rm loc}(\ov \Omega)$ by
	Assumption~\ref{asm:a.r}, thus $(-\Delta +q)\psi \in L^2_{\rm loc}(\ov \Omega)$ as well. We first prove the core property by a suitable
	cut-off, \cf~\cite[Proof of Thm.~8.2.1]{Davies-1995}. Let $\varphi$ be a $\CcRd$ function taking on non-negative values
	such that $\varphi(x) =1$ if $|x|<1$ and $\varphi(x)=0$ if $|x|>2$.
	For $\psi \in \Dom(H_\gamma^{\rm r})$ define
	\begin{equation}
	\psi_n(x) := \psi(x) \varphi_n(x), \quad \varphi_n(x):=\varphi \left(\frac xn \right), \quad x \in \Omega, \quad n \in \N.
	\end{equation}
	From the derived regularity of $\psi$ and the compactness of $\supp \varphi_n$, we conclude that $\{\psi_n\} \subset \cD$. Moreover,
	by the dominated convergence theorem, $\|\psi_n - \psi\| \to 0$ as $n \to \infty$, and 
	\begin{equation}
	\|H_\gamma^{\rm r}(\psi - \psi_n)\|
	\leq 
	\|(1-\varphi_n)(-\Delta + q + \gamma \ar)\psi\| +\| 2 \nabla \psi . \nabla \varphi_n + \psi \Delta \varphi_n\| \lto 0, \quad n \to \infty,
	\end{equation}
	since $\|\nabla \varphi_n\|_{\LiRd} = \frac 1n \|\nabla \varphi\|_{\LiRd}$ and $\|\Delta \varphi_n\|_{\LiRd} = \frac{1}{n^2} \|\Delta \varphi\|_{\LiRd}$.
	
	Next, we prove \eqref{norm.eq.1}. The second inequality in~\eqref{norm.eq.1} is obvious. 
	To prove the first one, we consider the cases $\Im \gamma \neq 0$ and $\gamma \neq 0$ only, the symmetric case with $\gamma > 0$ 
	being analogous and, in fact, simpler.  For every $\psi \in \cD$,
	\begin{equation}
	\begin{aligned}
	\|H_\gamma^{\rm r} \psi\|^2  & = \|(-\Delta +q) \psi\|^2 + |\gamma|^2 \|\ar \psi \|^2 + 2 \Re \langle (-\Delta+q) \psi, \gamma \ar \psi \rangle
	\\ 
	& = \|(-\Delta +q) \psi\|^2 + |\gamma|^2 \|\ar \psi \|^2 + 2 (\Re \gamma) \langle q^\frac12 \psi,\ar q^\frac12\psi \rangle
	\\  & \quad + 2 \Re (\gamma \langle \nabla \psi, \nabla (\ar \psi) \rangle);
	\end{aligned}
	\end{equation}
	note that the second step is justified since it can be verified that $\ar\psi \in \Dom(h_0)$. 
	Straightforward manipulations with the last term yield that
	\begin{equation}
	\begin{aligned}
	2 \Re (\gamma \langle \nabla \psi, \nabla (\ar \psi) \rangle) & = 2 (\Re \gamma) \langle \nabla \psi, \ar \nabla \psi \rangle + 2 \Re (\gamma \langle \nabla \psi, \psi \nabla \ar\rangle)
	\\
	& \geq 
	2 (\Re \gamma) \langle \nabla \psi,  \nabla (\ar \psi) \rangle
	- 4 |\gamma| \langle |\psi | |\nabla \ar|, |\nabla \psi|\rangle. 
	\end{aligned}
	\end{equation}
	Hence, for every $\eps_1 \in (0,1)$,
	\begin{equation}
	\begin{aligned}
	\|H_\gamma^{\rm r} \psi\|^2  
	& \geq \|(-\Delta +q) \psi\|^2 + |\gamma|^2 \|\ar \psi \|^2 + 2 (\Re \gamma) \langle (-\Delta + q) \psi,\ar \psi \rangle 
	\\ & \quad 
	- 4 |\gamma| \, \langle |\psi | |\nabla \ar|, |\nabla \psi|\rangle
	\\
	& \geq
	\frac{\eps_1}{1+\eps_1}  \|(-\Delta +q) \psi\|^2 + ((\Im \gamma)^2 - \eps_1(\Re \gamma)^2) \|\ar \psi \|^2 
	\\
	& \quad - 4 |\gamma| \, \langle |\psi | |\nabla \ar|, |\nabla \psi|\rangle
	\end{aligned}
	\end{equation}
	where we used Young's inequality in the last step.
	Since $\psi \in \Dom(H_0)$ and $\ar$ satisfies \eqref{asm:a.reg}, we see that, for every $\eps_2, \eps_3 \in (0,1)$,
	\begin{equation}\label{Gr.norm.mix}
	\begin{aligned}
	2  \langle |\psi ||\nabla \ar|, |\nabla \psi| \rangle & 
	\leq 2\langle (\eps_2 \ar^\frac 32 + M_\nabla(q^\frac12 +1))|\psi|, |\nabla \psi|\rangle
	\\
	&
	\leq \eps_2( \|\ar^\frac 12 \nabla \psi\|^2 + \|\ar\psi\|^2) 
	\\ & \quad+ 2 M_\nabla (\|\nabla \psi \|^2+ \| q^\frac12 \psi\|^2 +  \|\psi\|^2)
	\\
	& \leq \eps_2( \|\ar^\frac 12 \nabla \psi\|^2 + \|\ar\psi\|^2)  + \eps_3 \|(-\Delta +q)\psi\|^2 + C \|\psi\|^2
	\end{aligned}
	\end{equation}
	where $C$ is independent of $\psi$.
	Combining the estimates above, we obtain
	\begin{equation}\label{Gr.norm.mid}
	\begin{aligned}
	\|H_\gamma^{\rm r} \psi\|^2  
	& \geq
	\left(\frac{\eps_1}{1+\eps_1}  - 2|\gamma|\eps_3 \right) \|(-\Delta+q)\psi\|^2 
	\\
	& \quad +
	((\Im \gamma)^2 - \eps_1(\Re \gamma)^2 - 2 \eps_2|\gamma|) \|\ar \psi \|^2 
	\\
	& \quad - 2 \eps_2|\gamma| \, \|\ar^\frac 12 \nabla \psi\|^2 - 2 |\gamma|C \|\psi\|^2.
	\end{aligned}
	\end{equation}
	It remains to consider the term $\|\ar^\frac 12 \nabla \psi\|^2$ in~\eqref{Gr.norm.mid}. Clearly, we have
	\begin{equation}\label{Gr.norm.a1}
	2 \Im \langle H_\gamma^{\rm r} \psi, \sgn(\Im \gamma)(-\Delta + q) \psi \rangle \leq  \frac{1}{\eps_4} \|H_\gamma^{\rm r} \psi\|^2 + \eps_4 \|(-\Delta+q)\psi\|^2
	\end{equation}
	for any $\eps_4 \in (0,1)$.
	On the other hand,
	\begin{equation}\label{Gr.norm.a2}
	\begin{aligned}
	\Im \langle H_\gamma^{\rm r} \psi, \sgn(\Im \gamma)(-\Delta + q) \psi \rangle &=  
	\Im  (\gamma \sgn(\Im \gamma) \langle \ar \psi, (-\Delta+q)\psi \rangle) 
	\\
	& \geq  |\Im \gamma| \, \|\ar^\frac 12 \nabla \psi\|^2 -  |\gamma| \, \langle |\nabla \ar| |\psi|,|\nabla \psi|\rangle.
	\end{aligned}
	\end{equation}
	Thus using \eqref{Gr.norm.a1}, \eqref{Gr.norm.a2} in \eqref{Gr.norm.mid} and \eqref{Gr.norm.mix}, we arrive at
	\begin{equation}
	\begin{aligned}
	\left( 1+ \frac{1}{\eps_4 } \right) \|H_\gamma^{\rm r} \psi\|^2  
	&\geq  
	\left(\frac{\eps_1}{1+\eps_1}  -3  |\gamma| \eps_3 - \eps_4 \right) \|(-\Delta+q)\psi\|^2 
	\\
	& \quad +
	\left( (\Im \gamma)^2 - \eps_1(\Re \gamma)^2 -  3 |\gamma| \eps_2 \right) \|\ar \psi \|^2 
	\\
	& \quad + ( 2|\Im \gamma| -  3 |\gamma| \eps_2 ) \|\ar^\frac 12 \nabla \psi\|^2 - 3 |\gamma| C \|\psi\|^2.
	\end{aligned}
	\end{equation}
	Hence, we can successively select $\eps_1, \eps_2, \eps_3, \eps_4 \in (0,1)$ such that the coefficients of the first three terms
	are positive. Then a standard argument shows the existence of $k_1>0$, \cf~for instance~\cite[Proof of Lem.~2.9]{Boegli-2017-42},
	as required in \eqref{norm.eq.1}.

	The $\cC$-self-adjointness of $H_\gamma$ is straightforward by the representation
	theorem~\cite[Thm.~VI.2.1]{Kato-1966} and thus also \eqref{T.Jsa} follows.
\end{proof}

\begin{remark}\label{rem:q.reg}
	If $q$ satisfies certain regularity assumptions similar to those for $a$, then also
	\begin{equation}
	\Dom(-\Delta+q) = \Dom(\DD) \cap \Dom(q).
	\end{equation} 
	The latter holds \eg~if there is a decomposition $q =\qr+\qs$ with $\qr \geq 0$, $\qr \in W^{1,\infty}_{\rm loc}(\ov \Omega)$,
	$\qs \in \LtlocOm$ and, for each $\eps>0$, there are constants $M_{\nabla,q}=M_{\nabla,q}(\eps) \ge 0$ and
	\mbox{$M_{{\rm s},q} = M_{{\rm s},q}(\eps) \ge 0$}
	such that
	\begin{equation}
	|\nabla \qr| \leq \eps \qr^\frac 32 + M_{\nabla,q}
	\end{equation}
	and, for all $\psi \in \Dom(\DD) \cap \Dom(\qr)$,
	\begin{equation}\label{asm:q2}
	\|\qs \psi\| \leq \eps (\|\DD \psi\|+\|\qr \psi\|) + M_{{\rm s},q} \|\psi\|. 
	\end{equation}
	The proof is a simpler version of the proof of Theorem~\ref{thm:dom.H.gamma}.
\end{remark}

\begin{proof}[Proof of Theorem~\upshape \ref{thm:m-ac}]
	Using integration by parts, it is straightforward to check that, for all $\Phi:=(\phi_1,\phi_2) \in \Dom(G_0)$,
	\begin{equation}
	\langle G_0 \Phi, \Phi \rangle_\H
	= 2 \ii \Im 
	\big(
	\langle \nabla \phi_2, \nabla \phi_1 \rangle + \langle q^\frac 12 \phi_2, q^\frac 12 \phi_1 \rangle
	\big)
	- 2 \| a^\frac 12 \phi_2 \|^2.
	\end{equation}
	Thus $\Num (-G_0) \subset \{ z \in \C \, : \, \Re z \geq 0 \}$ and so $G_0$ is closable by \cite[Thm.~V.3.4]{Kato-1966}.
	
	Let $\cD$ be the core of $(-\Delta+q)$ defined in \eqref{core.def}. We prove that $\cD \times \CcOm \subset \Ran(G_0-1)$. To
	this end, we take an arbitrary $\Psi:=(\psi_1,\psi_2) \in \mathcal D \times \CcOm$ and find a solution
	$\Phi:= (\phi_1,\phi_2) \in \Dom(G_0)$ of $(G_0 -1) \Phi = \Psi$, \ie~of the system
	\begin{equation}
	\begin{aligned}
	- \phi_1 + \phi_2 & = \psi_1, \\
	(\Delta - q) \phi_1 -(2a + 1) \phi_2 & = \psi_2.
	\end{aligned}
	\end{equation}
	Solving the first equation for $\phi_2$ and inserting this into the second equation, we get
	\begin{equation}
	(-\Delta + q + 2 a + 1) \phi_1 = - (\psi_2 + (2a + 1) \psi_1).
	\end{equation}
	Note that the left hand side equals $T(1) \phi_1$ with $T(\la)$ defined in Section~\ref{sec:T.def}, \cf~\eqref{T.def}. Moreover, for $\la=1$, $\Dom(T(1)) = \Dom(-\Delta+q) \cap \Dom(a)$, \cf~Theorem~\ref{thm:dom.H.gamma}, and $0 \notin \sigma (T(1))$ since $T(1)$ is uniformly positive. Thus $T(1)^{-1}$ is a bounded operator in $\LOm$ and hence we obtain the solution $\Phi= (\phi_1,\phi_2)$,
	\begin{equation}
	\phi_1 = - T(1)^{-1}(\psi_2 + (2a + 1) \psi_1), \qquad \phi_2= \psi_1 + \phi_1. 
	\end{equation}
	Since $\psi_1 \in \cD \subset \Dom(-\Delta+q)$ and $\supp \psi_1$ is compact, we have $a \psi_1 \in \LOm$
	due to \eqref{asm:a2} and $\ar \in L^\infty_{\loc}(\ov \Omega)$. By Theorem~\ref{thm:dom.H.gamma} and because
	$\psi_2 \in \CcOm$, we see that $\phi_1 \in \Dom(T(1))$ and thus $\phi_2 \in \Dom(T(1))$ since $\psi_1 \in \cD \subset \Dom(T(1))$.
	Altogether this proves $\Phi \in \Dom(G_0)$.
\end{proof}

\section{Spectral equivalence for the generator $G$ and the associated quadratic function $T$}
\label{sec:sp.Gt}

In this section we prove spectral equivalence for the generator $G$ and the quadratic operator function $T$. 
To this end, we first derive some basic spectral properties of the operator family $T(\la)$, $\la \in \Cla$. 
\begin{proposition}
\label{prop:sp.T}
Let $a,q \in \LolocOmR$ and $a,q \geq 0$, let $T(\la)$, $\la \in \Cla$, be as in \eqref{T.def} and let $\ai:=\essinf(a)$. Then, for every $\la \in \Cla$,
\begin{enumerate}[{\upshape i)}]
	\item \label{prop:sp.T.i} $0 \in \sigma(T(\la)) \ \Longleftrightarrow \ 0 \in \sigma(T(\ov \la))$,
	\item \label{prop:sp.T.ii} $0 \in \sigma(T(\la)) \ \Longrightarrow \  \Re \la \leq - \ai \ \text{and} \ |\la|^2 \geq \inf \sigma(-\Delta +q)$,
	\item \label{prop:sp.T.iii} if, in addition, Assumption \ref{asm:a.gr} is satisfied, then, for all $\la \in \Cla$, 
	\begin{equation}
	0 \in \sigma(T(\la)) \ \Longleftrightarrow \ 0 \in \spd(T(\la))
	\end{equation}
	and the set $\{ \la \in \Cla \, : \, 0 \in \sigma(T(\la))\}$ consists only of isolated points which may accumulate at most at $(-\infty,0]$.
\end{enumerate}
\end{proposition}
\begin{proof}
	i) The claim is immediate from $T(\la)^* = T(\ov \la)$, $\la \in \Cla$, \cf~\eqref{T.Jsa}.
	
	ii) We rely on a numerical range argument. Recall that $T(\la)$ is defined through the sectorial form $\widetilde h_{2\la}$ and so for
	$0 \in \sigma(T(\lambda))$ it is necessary that $0 \in \ov{\Num(T(\la))} = $ $
	\e^{\ii \arg(\la)/2}\Num(\widetilde h_{2\la}) + \lambda^2$, \cf~Lemma~\ref{lem:h.def}. 
	But this is impossible if $\Re \la >0$ by the enclosure \eqref{h.gam.nr} with $\gamma=2\la$.
	We proceed further by contradiction. Let $\la \in \Cla$ with $\Re \la \leq 0$ be such that $0 \in \sigma(T(\la))$ and 
	$\ai + \Re \la > \eps >0$ or $\inf \sigma(-\Delta +q)-|\la|^2 > \eps >0$. By the numerical range argument above, there is a
	sequence $\{z_n\} \subset \e^{\ii \arg(\la)/2}\Num(\widetilde h_{2\la}) + \lambda^2$ such that $z_n \to 0$. Then there is a sequence
	$\{\psi_n\} \subset \Dom(\widetilde h_{2\la})$, $\|\psi_n\|=1$, such that
	\begin{equation}\label{nr.zn}
	\|\nabla \psi_n\|^2 + \|q^\frac 12 \psi_n\|^2 + 2\la \|a^\frac 12 \psi_n\|^2 + \la^2 \|\psi_n\|^2 = z_n.
	\end{equation}
	Taking the real and imaginary part of \eqref{nr.zn}, we find 
	\begin{align}
	\|\nabla \psi_n\|^2 + \|q^\frac 12 \psi_n\|^2 + 2 \Re\la \|a^\frac 12 \psi_n\|^2 + (\Re\la)^2 - (\Im\la)^2  &= \Re z_n,
	\label{nr.zn.re}
	\\
	2 \Im\la \big( \|a^\frac 12 \psi_n\|^2 +  \Re\la  \big)  &= \Im z_n.
	\label{nr.zn.im}
	\end{align}
	Recall that $\Im \la \neq 0$ since $\la \in \Cla$ with $\Re \la \leq 0$.
	
	First consider the case when $\ai + \Re \la > \eps >0$. Then \eqref{nr.zn.im} yields
	\begin{equation}
	\frac{|\Im z_n|}{2 |\Im \la|} =  \|a^\frac 12 \psi_n\|^2 +  \Re\la  > \eps >0,
	\end{equation}
	a contradiction to $z_n \to 0$.
	
	In the second case when $\inf \sigma(-\Delta +q)-|\la|^2 > \eps >0$, we solve \eqref{nr.zn.im} for $\|a^{1/2} \psi_n\|^2$ 
	and insert this into \eqref{nr.zn.re} to obtain
	\begin{equation}\label{nr.zn.eq}
	\|\nabla \psi_n\|^2 + \|q^\frac 12 \psi_n\|^2 - |\la|^2 = \Re z_n - \frac{\Re \la}{\Im \la} \Im z_n.  
	\end{equation}
	Since the minimum of the spectrum of a self-adjoint operator coincides with the infimum of its numerical range and by the assumption on $|\la|$, we have
	\begin{equation}\label{nr.inf}
	\|\nabla \psi_n\|^2 + \|q^\frac 12 \psi_n\|^2 - |\la|^2 
	\geq \inf_{\psi \in \Dom(\widetilde h_{2\la}), \|\psi\|=1} (\|\nabla \psi\|^2 + \|q^\frac 12 \psi\|^2) - |\la|^2  > \eps.
	\end{equation}
	Inserting \eqref{nr.inf} into \eqref{nr.zn.eq}, we again arrive at a contradiction to $z_n \to 0$.
	
	iii) The claim follows from \cite[Thm.~VII.1.10]{Kato-1966} if we show that $T(\la)$, $\la \in \Cla$, is a holomorphic family of closed operators in $\LOm$
	with compact resolvent and that there is a $\la_0 \in \C \setminus (-\infty,0]$ for which $T(\la_0)^{-1}$ exists and is bounded in $\LOm$. The compactness of the
	resolvents is proved in Lemma~\ref{lem:h.def} \ref{H.comp.res} and $T(\la)$ is holomorphic since $H_{2\lambda}$ is holomorphic,
	\cf~Lemma~\ref{lem:h.def} \ref{H.hol.it}, and $\la^2$, viewed as a multiplication operator, is a bounded holomorphic family, \cf~\cite[Prob.~VII.1.2]{Kato-1952-125}.
	Since, for any $\la_0>0$, $T(\la_0)$ is a uniformly positive operator, we can choose $\la_0\in(0,\infty)$ arbitrarily.
\end{proof}

In the case where the spectrum is discrete and there are no real eigenvalues, it is possible to extend Proposition~\ref{prop:sp.T}.\ref{prop:sp.T.ii} and
derive further estimates on the absolute values of eigenvalues for quadratic pencils, \cf~\cite{Freitas-1999-78} for the matrix case and also for wave
equations on bounded domains with bounded damping via discretization.

\begin{theorem}\label{thm:GT}
	Let $a,q$ satisfy Assumption \ref{asm:a.r}, and let $G$, $T(\la)$ be as in \eqref{G.def}, \eqref{T.def}, respectively. Then, for all $\la \in \Cla$,
	\begin{equation}\label{sp.TG.eq}
	\begin{aligned}
	\la \in \sigma(G) \quad &\Longleftrightarrow \quad 0 \in \sigma(T(\la)),
	\\
	\la \in \spp(G) \quad&\Longleftrightarrow \quad 0 \in \spp(T(\la)),
	\\
	\la \in \se{2}(G) \quad&\Longleftrightarrow \quad 0 \in \se{2}(T(\la)),
	\end{aligned}
	\end{equation}
	and
	\begin{equation}
	\psi \in \Ker(T(\la)) \iff (\psi,\la \psi) \in \Ker(G-\la).
	\end{equation}
	If, in addition, $a$ satisfies Assumption \ref{asm:a.gr}, then $\sigma(G) \cap \Cla$ consists only of eigenvalues of finite multiplicity which may only accumulate at $(-\infty,0]$.
\end{theorem}

\begin{proof}
	Let $\la \in \Cla$ be fixed. We split the proof into several steps.
	
	\noindent
	$\bullet$ Claim i): $\Psi=(\psi_1,\psi_2) \in \Ran(G-\la)^\perp \Longrightarrow \ov \la \psi_1 + \psi_2 = 0$ and $\psi_2 \in \Ker(T(\la)^*)$.  
	
	To see this, take $\Phi=(\phi_1,\phi_2) \in \Dom(G_0)$ with $G_0$ as in \eqref{G0.def}. Then we have $\langle (G_0-\la) \Phi,\Psi \rangle_\cH=0$ or, equivalently, \cf~\eqref{Hilb.space}, 
	\begin{equation}\label{Ran.perp.G0}
	\langle \nabla (\phi_2 - \la \phi_1), \nabla \psi_1 \rangle + \langle q^\frac12 (\phi_2 - \la \phi_1), q^\frac12 \psi_1 \rangle + \langle (\Delta-q) \phi_1 - (2a+\la) \phi_2, \psi_2 \rangle=0.
	\end{equation}
	If we set $\phi_2=\la \phi_1$, we get $\langle T(\la) \phi_1, \psi_2 \rangle =0$ for all $\phi_1 \in \Dom(T(\la))$. Hence, $\psi_2 \in \Dom(T(\la)^*) = \Dom(T(\la))$ and $T(\la)^* \psi_2 =0$.
	On the other hand, if we set $\phi_2=0$, then \eqref{Ran.perp.G0} and $\psi_2 \in \Dom(T(\la)^*) \subset \Dom(h_0)$, \cf~Theorem~\ref{thm:dom.H.gamma}, imply that
	\begin{equation}
	\langle \nabla \phi_1, \nabla (\ov \la \psi_1 + \psi_2) \rangle + \langle q^\frac12 \phi_1, q^\frac12 (\ov \la \psi_1 + \psi_2) \rangle =0.
	\end{equation}
	Since $\Dom(-\Delta + q) \cap \Dom(a)$ is dense in $\mathcal W(\Omega)$, we obtain $\ov \la \psi_1 + \psi_2 =0$.
	
	\noindent
	$\bullet$ Claim ii): $\Psi=(\psi_1,\psi_2) \in \Ker(G-\la) \Longleftrightarrow \la \psi_1 - \psi_2 = 0$ and $\psi_2 \in \Ker(T(\la))$.  
	
	It is straightforward to check the implication ``$\Longleftarrow$'' since the assumptions imply that
	$\Psi \in \Dom(G_0)$, \cf~Theorem~\ref{thm:dom.H.gamma} and \eqref{G0.def}. To prove the implication ``$\Longrightarrow$'', we first integrate by parts to conclude that the operator
	\begin{equation}
	G_0^{\rm c} := 
	\begin{pmatrix}
	0 & -I \\
	-\Delta + q & -2a
	\end{pmatrix}, 
	\quad
	\Dom(G_0^{\rm c}) := \cD \times \cD,
	\end{equation}
	with $\cD$ defined as in \eqref{core.def} is a densely defined restriction of $G^* = G_0^*$. Then $\Psi \in \Ker(G-\la)$ implies that, for all $\Phi = (\phi_1,\phi_2) \in \Dom(G_0^{\rm c})$, 
	\begin{equation}
	0=\langle (G-\la)\Psi, \Phi \rangle_\cH =\langle \Psi, (G_0^{\rm c} - \ov \la) \Phi \rangle_\cH
	\end{equation}
	or, equivalently,
	\begin{equation}\label{Ker.G}
	-\langle \nabla \psi_1, \nabla (\ov \la \phi_1 + \phi_2) \rangle - \langle q^\frac12 \psi_1, q^\frac12 (\ov \la \phi_1 + \phi_2) \rangle - \langle \psi_2, (\Delta-q) \phi_1 + (2a+\ov \la) \phi_2 \rangle=0.
	\end{equation}
	Setting $\phi_2=-\ov \la \phi_1$, we obtain $\langle \psi_2, T(\la)^* \phi_1 \rangle=0$ for all $\phi_1 \in \cD$. Since $\cD$ is a core of $T(\la)^*$, we have $\psi_2 \in \Dom(T(\la))$ and $T(\la)\psi_2=0$. Finally, setting $\phi_2=0$ and using \eqref{Ker.G}, we find that, for all $\phi _1 \in \cD$,
	\begin{equation}
	\langle \nabla (\la \psi_1-\psi_2), \nabla \phi_1 \rangle + \langle q^\frac12 (\la \psi_1-\psi_2), q^\frac12 \phi_1 \rangle =0,
	\end{equation}
	hence $\la \psi_1-\psi_2 =0$ because $\cD$ is dense in $\mathcal W(\Omega)$.
	
	\noindent
	$\bullet$ Claim iii): $0 \in \se{2}(T(\la)) \Longleftrightarrow \la \in \se{2}(G)$.
	
	Let $0 \in \se{2}(T(\la))$ and let $\{\psi_n\} \subset \Dom(T(\la))$ be a corresponding singular sequence, \ie~$\|\psi_n\|=1$, $\psi_n \xrightarrow{w} 0$ and $T(\la) \psi_n \to 0$
	in $\LOm$ as $n \to \infty$.
	Then $\Psi_n:=(\psi_n, \la \psi_n) \in \Dom(G_0)$, $n \in \N$, and 
	\begin{equation}
	\frac{\|(G_0-\la)\Psi_n\|_\cH}{\|\Psi_n\|_\cH} \leq \frac{\|T(\la) \psi_n\|}{|\la| } \to 0, \quad n \to \infty.
	\end{equation}
	Thus it remains to be verified that $\widetilde \Psi_n:=\Psi_n/\|\Psi_n\|_\cH \xrightarrow{w} 0 $ as $n \to \infty$ in $\cH$. Since $\|\widetilde \Psi_n\|_\cH=1$, it suffices
	to check weak convergence on $\cD \times \cD$ which is dense in $\cH$. Indeed, for $\Phi=(\phi_1,\phi_2) \in \cD \times \cD$,
	\begin{equation}
	|\langle \widetilde \Psi_n, \Phi \rangle_\cH| \leq \frac{|\langle \psi_n, (-\Delta+q) \phi_1 \rangle| + |\la| |\langle \psi_n, \phi_2 \rangle|}{|\la|} \to 0, \quad n \to \infty,
	\end{equation}
	since $\psi_n \xrightarrow{w} 0 $ in $\LOm$ as $n \to \infty$. Hence the implication ``$\Longrightarrow$'' is proved.
	
	To prove the reverse implication ``$\Longleftarrow$'', assume that $0 \notin  \se{2}(T(\la))$. In order to show that $\la \notin \se{2}(G)$, we construct a (bounded)
	left approximate inverse, \cf~\cite[Def.~I.3.8]{EE}, of $G-\la$. Then it follows from \cite[Thm.~I.3.13]{EE} that $G-\la$ is semi-Fredholm. 
	Moreover, we have $\dim \Ker(G-\la)<\infty$ by claim ii) proved above.

	It remains to construct a left approximate inverse of $G-\la$. Since $T(\la)$ is $J$-self-adjoint, we have $\dim \Ker(T(\la)) = \dim \Ker(T(\la)^*)$, 
	\cf~\cite[Lem.~III.5.4]{EE}, thus $T(\la)$ is Fredholm. Hence there exists a generalized inverse $T(\la)^\#$, 
	\cf~\cite[Sec.~5]{Nashed-1976},~\ie
	\begin{equation}\label{T.dag.id}
	\begin{aligned}
	T(\la) T(\la)^\#\psi &= \psi - Q \psi, \quad \psi \in \LOm,
	\\
	T(\la)^\# T(\la)  \psi &= \psi - P \psi, \quad \psi \in \Dom(T(\la)),
	\end{aligned}
	\end{equation}
	where $P,Q$ are the orthogonal projections on $\Ker(T(\la))$, $\Ker(T^*(\la))$, respectively.
	
	Let $\Phi=(\phi_1,\phi_2) \in \Dom(G_0)$ and $\Psi=(\psi_1,\psi_2) \in \cH$ be so that $(G_0-\la)\Phi=\Psi$,~\ie
	\begin{equation}
	\begin{aligned}
	\phi_2 - \la \phi_1 & = \psi_1, \\
	(\Delta-q) \phi_1 - (2a +\la) \phi_2 & = \psi_2;
	\end{aligned}
	\end{equation}
	notice that $\psi_1 \in \Dom(-\Delta+q)$ by the first equation and since $\Phi \in \Dom(G_0)$.
	Solving the first equation for $\phi_1$, \ie~$\phi_1 = \la^{-1}(\phi_2 - \psi_1)$, and inserting this expression into the second equation, we obtain,
	after multiplication by $\la$,
	\begin{equation}
	T(\la) \phi_2 = (-\Delta + q) \psi_1 -\la \psi_2. 	
	\end{equation}
	Applying the generalized inverse $T(\la)^\#$, we find
	\begin{equation}
	\phi_2 = T(\la)^\# (-\Delta + q) \psi_1 - \la T(\la)^\# \psi_2 + P \phi_2, 
	\end{equation}
	and thus, recalling that $\phi_1 = \la^{-1}(\phi_2 - \psi_1)$, we arrive at
	\begin{equation}\label{Rla.def}
	\hspace{-0.051cm}
	\begin{pmatrix}
	\phi_1 \\ \phi_2
	\end{pmatrix}
	=
	\underbrace{
		\left(	
		\begin{array}{lr}
		\frac 1 \la 
		\left(
		T(\la)^\# (-\Delta + q) -I \right) & \hspace{-0.2cm} -T(\la)^\#
		\\[1mm]
		\quad \ T(\la)^\# (-\Delta + q) & \hspace{-0.2cm} - \la T(\la)^\#
		\end{array}
		\right)}_{=:\widehat R_\la}
	\begin{pmatrix}
	\psi_1 \\ \psi_2
	\end{pmatrix}
	+
	\underbrace{
		\begin{pmatrix}
		0	& \frac 1 \la P
		\\
		0 & P
		\end{pmatrix}
	}_{=:K_\la}
	\begin{pmatrix}
	\phi_1 \\ \phi_2
	\end{pmatrix}.
	\end{equation}
	Hence, for all $\Phi \in \Dom(G_0)$, $\widehat R_\la (G_0 -\la) \Phi = \Phi - K_\la \Phi$
	and $K_\la$ is compact since $P$ has finite rank and is everywhere defined, both as an operator in $\LOm$ and as an operator from $\LOm$ to $\mathcal W(\Omega)$
	because $\Ker(T(\la)) \subset \Dom(T(\la)) \subset \mathcal W(\Omega)$.

	Next we show that $\widehat R_\la$ has a bounded extension $R_\la$ onto $\H$, which is a left approximate inverse for the closure $G-\la$ of $G_0-\la$, \ie~, 
	\begin{equation}\label{G.li}
	R_\la (G -\la) \Phi = \Phi - K_\la \Phi, \quad \Phi \in \Dom(G).
	\end{equation}
	To this end, in the representation of $\widehat R_\la$, \cf~\eqref{Rla.def}, we replace $T(\la)^\#$ first by $(T(\la)+\la_0)^{-1}$ with some
	$\la_0 \in \rho(T(\la)) \neq \emptyset$ and then the latter by the self-adjoint operator $T(1)^{-1}$. More precisely, with the help of \eqref{T.dag.id}, we derive the resolvent-type identities
	\begin{equation}
	\begin{aligned}
	T(\la)^\# &= T(\la)^\# (T(\la)+\la_0)(T(\la)+\la_0)^{-1} 
	\\ &
	= 
	(I-P)(T(\la)+\la_0)^{-1}  + \la_0 T(\la)^\# (T(\la)+\la_0)^{-1},
	\\
	T(\la)^\# &= (T(\la)+\la_0)^{-1}(T(\la)+\la_0) T(\la)^\#
	=(T(\la)+\la_0)^{-1} (I-Q + \la_0 T(\la)^\#),
	\end{aligned}
	\end{equation}
	hence
	\begin{equation}
	T(\la)^\# = (I-P)(T(\la)+\la_0)^{-1} +\la_0  (T(\la)+\la_0)^{-1} (I-Q + \la_0 T(\la)^\#) (T(\la)+\la_0)^{-1}.
	\end{equation}
	Similarly,
	\begin{equation}
	\begin{aligned}
	(T(\la)+\la_0)^{-1} &= T(1)^{-1} - (T(\la)+\la_0)^{-1} (2(\la-1) a + \la^2 -1 +\la_0 ) T(1)^{-1}, 
	\\ 
	(T(\la)+\la_0)^{-1} &= T(1)^{-1} - T(1)^{-1} (2(\la-1)a + \la^2-1+\la_0)    (T(\la)+\la_0)^{-1}.
	\end{aligned}
	\end{equation}
	Since $\Dom(T(\la)) = \Dom(T(\la)^*) = \Dom(-\Delta+q) \cap \Dom(a)$ for all $\la \in \Cla$, the composition $a(T(\la)+\la_0)^{-1}$ is bounded on
	$\LOm$; since $(T(\la)+\la_0)^{-1}a \subset (a ((T(\la)+\la_0)^{-1})^*)^*$, the operator $(T(\la)+\la_0)^{-1}a$ has a bounded extension onto $\LOm$.
	
	A careful inspection of the individual terms in $\widehat R_\la$ using the identities derived for $T(\la)^\#$ shows that the
	most problematic term is $T(1)^{-1}(-\Delta+q)$; we will show that it has an extension to a bounded operator from
	$\mathcal W (\Omega )$ to $\mathcal W (\Omega )$. The remaining terms can be handled in a similar (simpler) way;
	notice also that the
	terms containing $P$ or $Q$ are of finite rank and everywhere defined since $\Dom(T(\la)) = \Dom(T(\la)^*) \subset \mathcal W(\Omega)$.
	
	Now let $\phi \in \Dom (-\Delta+q)$. Then, using the second representation theorem \cite[Thm.~VI.2.23]{Kato-1966} for $-\Delta+q$ and denoting
	$\psi :=(-\Delta+q)^{1/2} \phi$, we obtain
	\begin{equation}
	\begin{aligned}
	&\frac{\|\nabla T(1)^{-1} (-\Delta+q) \phi\|^2 + \|q^\frac 12 T(1)^{-1} (-\Delta+q) \phi\|^2}{\|\nabla \phi\|^2 + \|q^\frac 12 \phi\|^2}
	\\
	& \qquad \qquad = 
	\frac{\|(-\Delta+q)^{\frac12}T(1)^{-1}(-\Delta+q) \phi\|^2}{\|(-\Delta+q)^{\frac12} \phi\|^2}
	\\
	& \qquad \qquad = 
	\frac{\|(-\Delta+q)^{\frac12}(-\Delta+q+2a +1)^{-1}(-\Delta+q)^{\frac12} \psi\|^2}{\|\psi\|^2}.
	\end{aligned}
	\end{equation}
	Since
	\begin{equation}\label{T.bdd.ext}
	\begin{aligned}
	&(-\Delta+q)^{\frac12}(-\Delta+q+2a +1)^{-1}(-\Delta+q)^{\frac12} 
	\\ & \quad 
	\subset (-\Delta+q)^{\frac12}(-\Delta+q+2a +1)^{-\frac12} ((-\Delta+q)^{\frac12} (-\Delta+q+2a +1)^{-\frac12})^*
	\end{aligned}
	\end{equation}
	and the operator on the right-hand side of~\eqref{T.bdd.ext} is bounded on $\LOm$, we have 
	\begin{equation}
	\begin{aligned}
	\frac{\|(-\Delta+q)^{\frac12}(-\Delta+q+2a +1)^{-1}(-\Delta+q)^{\frac12} \psi\|^2}{\|\psi\|^2} \leq M < \infty.
	\end{aligned}
	\end{equation}
	Hence $T(1)^{-1}(-\Delta+q)$ is bounded on a dense subset of $\mathcal W(\Omega)$, so it has a bounded extension on $\mathcal W(\Omega)$.
	
	\noindent
	$\bullet$ Claim iv): $0 \in \rho(T(\la)) \Longleftrightarrow \la \in \rho(G)$.
	
	The implication ``$\Longrightarrow$'' follows immediately from claims iii), i) and ii) since $\Ran(G-\la)$ is closed and $\dim \Ker(G-\la)=\dim \Ran(G-\la)^\perp =0$.
	To show the other direction, notice that if $0 \in \sigma(T(\la))$, then $0 \in \spp(T(\la))$ or $0 \in \se{2}(T(\la))$ since
	$T(\la)$ is $\cC$-self-adjoint, \cf~\eqref{T.Jsa}. Hence by claims ii) and iii), respectively, we have shown that then $\la \in \spp(G)$
	or $\la \in \se{2}(G)$.
	
	\noindent
	$\bullet$ Finally, if $a$ additionally satisfies Assumption \ref{asm:a.gr}, the last claim follows from the established equivalences \eqref{sp.TG.eq} and
	Proposition~\ref{prop:sp.T}.\ref{prop:sp.T.iii}.
\end{proof}
The following is a straightforward extension of the claim of Theorem~\ref{thm:dom.H.gamma} to $\Claq$ for some $a_q>0$; the details are left to the reader. 
\begin{remark}\label{rem:q.dom}
	Let the assumptions of Theorem~\ref{thm:GT} hold and let, in addition, $q$ satisfy the conditions in Remark~\ref{rem:q.reg} ensuring that
	$\Dom(-\Delta+q) = \Dom(\DD) \cap \Dom(q)$. If there are constants $k_1 \geq 0$, $k_2 \in \R$ such that
	\begin{equation}\label{qa.rb}
	a \leq k_1 q + k_2.
	\end{equation}
	and $M_q$ denotes the $q$-bound of $a$, \ie~the infimum of $k_1$ for which \eqref{qa.rb} holds, then the spectral equivalence \eqref{sp.TG.eq} holds
	for $ \la \in \Claq$ with
	\begin{equation}\label{alphaq.def}
	\alpha_q:=\frac1{2M_q} \in (0,+\infty].
	\end{equation}
	If, in addition, $a$ satisfies Assumption~\ref{asm:a.gr}, then $\sigma(G) \setminus (-\infty,\alpha_q]$ consists only of eigenvalues with
	finite multiplicity which may accumulate only at points in $(-\infty,-\alpha_q]$. 
\end{remark}

\section{Real essential spectrum of the generator $G$}
\label{sec:ess.sp}

In this section we investigate the essential spectrum of $G$ lying on the negative real semi-axis which is not accessible via the quadratic operator function $T(\la)$ since the latter is not defined for $\la \in (-\infty,0]$. Informally, if the underlying domain $\Omega$ contains a sufficiently
large neighborhood of a ray where the damping $a$ diverges as $|x| \to \infty$ and the potential $q$ does not dominate $a$, then $(-\infty,0] \subset \se{2}(G)$. We emphasize that we do not require the potential $q$ to be bounded. 

In the sequel we decompose $x\in \Rd$ as $x=(x_1,x')$ with $x_1 \in \R$ and $x' \in \R^{d-1}$. If $d>1$ and $\Omega \neq \R^d$, we suppose that $\Omega$ contains a ray $\Gamma:=\{(x_1,0) \, : \, x_1>0\}$ and a ``tubular'' neighborhood $U_\omega$ of $\Gamma$ given by
\begin{equation}\label{U.om.def}
U_\omega:= \left\{ (x_1, x') \in \Rd \,:\, x_1>0, \ |x'| < \omega(x_1)^{-1}  \right\}
\end{equation}
where $\omega: (0,\infty) \to (0,\infty)$ is a continuous function satisfying certain assumptions to be specified in Theorem~\ref{thm:Gspe} below. The radius $1/\omega(x_1)$ may shrink to $0$ at $\infty$, and the possible shrinking rate is controlled by the growth of the damping $a$. 
Note that for $d=1$, we may let $U_\omega = \Gamma$ and no function $\omega$ is needed.

We start with the simple observation that $0 \in \se{2}(G)$ if $\Omega$ contains a cone and $q$ decays therein as $|x| \to \infty$. 
\begin{proposition}\label{prop:q.dec}
	Let $a,q$ satisfy Assumption \ref{asm:a.r} and assume, in addition, that $q \in \LtlocOm$ and that $G$ is given
	by~\eqref{G.def}. If $\,\Omega$ contains a cone 
	\begin{equation}\label{cone.def}
	C_\delta:=\{ (x_1,x') \in \Rd \, : \, x_1>0, \ |x'|< \delta x_1  \}
	\end{equation}
	for some $\delta>0$ and if
	\begin{equation}\label{q.dec}
	\lim_{k \to \infty} \ \esssup_{x\in C_\delta, |x|>k}  \ q(x) = 0,
	\end{equation}
	then $0 \in \se{2}(G)$.	
\end{proposition}
\begin{proof}
It suffices to find a sequence $\{\Phi_n\} \subset \Dom(G_0)$, $\Phi_n \neq 0$, $n \in \N$, such that $\Phi_n/\|\Phi_n\|_\cH \wto 0$ in $\H$ and
$G_0 \Phi_n / \|\Phi_n\|_\cH \to 0 $ in $\H$ as $n \to \infty$. 

For $d>1$, we work in spherical coordinates $x= (|x|,\Theta)$ with $\Theta \in S^{d-1}$; the simplifications for $d=1$ are obvious.
Let $0\neq \varphi \in C_0^\infty((0,1))$, $0 \neq \chi \in C_0^\infty(S^{d-1} \cap C_\delta)$, and define 
\begin{equation}
\begin{aligned}
\phi_n(|x|,\Theta) := |x|^{-\frac{d-1}{2}} \varphi_n(|x|) \chi(\Theta), 
\quad
\varphi_n(|x|)  := \rho_n^\frac14 \varphi(\rho_n^\frac12 |x|-n), \quad  n \in \N,
\end{aligned}
\end{equation}
where
\begin{equation}
\rho_n := \esssup_{x\in C_\delta, |x|>n}  \ q(x).
\end{equation}
Straightforward, but lengthy, calculations yield that, as $ n \to \infty$,
\begin{align}
\|\varphi_n\|_{L^2(\R)} &= \BigO(1),& \|\varphi_n'\|_{L^2(\R)} &= \BigO(\rho_n^\frac12), &
 \|\varphi_n''\|_{L^2(\R)} &= \BigO(\rho_n),
\\
\|\nabla \phi_n\|^{-1} &= \BigO(\rho_n^{-\frac12}), &
\|\Delta \phi_n \| &= \BigO(\rho_n),  
& 
\|q \phi_n\| &= \BigO(\rho_n).
\end{align}
If we define $\Phi_n:=(\phi_n,0)$, then $\Phi_n/\|\Phi_n\|_\cH \wto 0$ as $n \to \infty$ since $\supp \phi_n$ moves to infinity. Using assumption~\eqref{q.dec},
we obtain 
\begin{equation}
\frac{\|G_0 \Phi_n \|_\cH^2}{\|\Phi_n\|^2_\cH}=
\frac{\|G_0 \Phi_n \|_\cH^2}{\|\nabla \phi_n\|^2 + \|q^\frac12 \phi_n\|^2} 
\leq
2 \frac{\|\Delta \phi_n\|^2 + \|q \phi_n\|^2}{\|\nabla \phi_n\|^2}
= \BigO(\rho_n) =\os(1),
\end{equation}
as $n \to \infty$.
\end{proof}

The following Theorem~\ref{thm:Gspe} provides conditions under which a fixed $\la \in (-\infty,0)$ belongs to $\se{2}(G)$. We remark that in the case where the damping $a$ dominates
the potential $q$ in a suitable $U_\omega$, \ie~$q(x) = \os(a(x))$ as $|x| \to \infty$ in $U_\omega$ (and the remaining regularity and growth conditions, then independent of $\la$, 
are satisfied),  \emph{every} $\lambda \in (-\infty,0)$ belongs to $\se{2}(G)$, hence $0 \in \sigma(G)$ as well. This effect is clearly visible in the examples, \cf~Section~\ref{sec:ex}. 
We also mention that for the simplest choice $\omega(x_1) = x_1^\alpha$, $x_1>0$, $\alpha \in \R$, the first two conditions in \eqref{om.reg} are satisfied
since, for $k \in \N$, $|\omega^{(k)}(x_1)| = \BigO(1/x_1^k) \omega(x_1)$ as $x_1 \to +\infty$.
\begin{theorem}\label{thm:Gspe}
Let $a,q$ satisfy Assumption \ref{asm:a.r} and assume, in addition, that $q \in \LtlocOm$, and that $G$ is 
defined as in~\eqref{G.def}. If, for $\la \in (-\infty,0)$, $\Omega$ contains a tubular neighborhood $U_\omega$ of a ray $\Gamma$ such that:

\begin{enumerate}[{\upshape i)}]
\item there is a decomposition
\begin{equation}
q(x) + 2\la a(x) + \la^2 = -A(x_1) + B(x), \quad x \in U_\omega,
\end{equation}
where $A \in C^1(\R_+)$,
\begin{equation}\label{A.lim}
\lim_{u \to + \infty} A(u) = + \infty, \quad
\lim_{u \to + \infty} \frac{|A'(u)|}{A(u)} = 0,
\end{equation}
and 
\begin{equation}\label{B.A.lim}
\lim_{n \to \infty} \ \esssup_{x \in U_\omega, |x|>n } \ \frac{|B(x)|^2}{A(x_1)} = 0\;
\end{equation}

\item if $d>1$, then $\omega \in C^2(\R_+)$ and
\begin{equation}\label{om.reg}
\frac{\omega'(u)}{\omega(u)} = \os(1), \quad  \frac{\omega''(u)}{\omega(u)^3 } = \BigO(1),  
\quad  \frac{\omega(u)^4}{A(u)} = \os(1), \quad u \to \infty,
\end{equation}
\end{enumerate}
then $\la \in \se{2}(G)$.

\end{theorem}
\begin{proof}
	Using a one-dimensional WKB expansion, we construct a singular sequence of the form
	$\{\Phi_n\} = \{(\phi_n,\la \phi_n)\} \subset \Dom(G_0) = \Dom(T(\la)) \times \Dom(T(\la))$ with
	$\supp \phi_n \subset U_\omega$ compact, $n\in\N$. We give a detailed proof for $d>1$; the simplifications for $d=1$ are obvious.
	The first components $\phi_n$ of $\Phi_n$ will be constructed such that
	\begin{equation}
	\frac{\|(-\Delta - A) \phi_n\|+ \|B\phi_n\|}{\|\nabla \phi_n\| } \to 0, \quad n \to \infty,
	\end{equation}
	and $\supp \phi_n$ moves to infinity in $U_\omega$; this implies
	\begin{equation}\label{sing.s.1}
	\begin{aligned}
	\|(G-\la) \Phi_n \|_\H &=\frac{\|T(\la) \phi_n\|}{\|\Phi_n\|_\cH}
	\leq 
	\frac{\|(-\Delta - A) \phi_n\|+ \|B\phi_n\|}{\|\nabla \phi_n\| } \to 0, \quad n \to \infty
	\end{aligned}
	\end{equation}
	and $\Phi_n/\|\Phi_n\|_\cH \wto 0$ as $n \to \infty$, respectively.

	By \eqref{A.lim}, we have $A(u)>0$ for all $u \in (\alpha,\infty)$ with some $\alpha>0$ and
	\begin{equation}\label{rho.def}
	\rho_n := \sup_{t>n} \frac{|A'(t)|}{A(t)} \to 0, \quad n > \alpha,  \quad n \to \infty.
	\end{equation}

	We write $x=(x_1,x') \in \Rd$ and denote by $\cB'$ the open $(d-1)$-dimensional unit ball. For $n \in \N$, we choose
	$\phi_n(x_1,x') := \varphi_n(x_1) \psi_\la (x_1) \chi(x)$ where
	\begin{align}
	\psi_\la(x_1)&:= \exp \left(\ii \int_{\alpha}^{x_1} A(t)^\frac 12 \dd t \right),&  & &&
	\\
	\chi(x) &:= \widetilde \chi(\omega(x_1) x'), & \widetilde \chi& \in C_0^{\infty}(\cB'), && \|\widetilde \chi\|_{L^2(\R^{d-1})}=1,
	\\
	\varphi_n(x_1) &:= \omega(x_1)^{\frac{d-1}{2}} \rho_n^{\frac14} \varphi \left(\rho_n^{\frac12} x_1 -n \right), & \varphi &\in C_0^{\infty}((0,1)),
	&& \|\varphi\|_{L^2(\R)}= \left(\int_{\cB'} \dd y' \right)^{-1} \! \! \! . \\[-8mm] \nonumber
	\end{align}
	Then $\supp \phi_n \subset \supp \chi  \subset  U_\omega$ and by the change of variables $(x_1,x')=(y_1,\omega^{-1}(y_1)y')$
	\begin{equation}
	\int_{U_\omega} |\varphi_n(x_1)|^2 \, \dd x = \int_0^\infty |\varphi(y_1)|^2 \, \dd y_1 \cdot \int_{\cB'} \, \dd y' = 1.
	\end{equation}
	Moreover, using the notation $\|f\|_{\infty,n}:=\esssup_{x \in \supp \varphi_n} |f(x)|$,  $n\in\N$, and the assumptions \eqref{om.reg} we obtain that, as $n \to \infty$,
	\begin{align}
	\int_{U_\omega} |\varphi_n'(x_1)|^2 \dd x 
	&= \BigO \left( \| \omega' \omega^{-1} \|^2_{\infty,n} + \rho_n\|\varphi'\|^2_{L^2(\R)} \right) = \os(1), 
	\\
	\int_{U_\omega} |\varphi_n''(x_1)|^2 \dd x 
	& = \BigO
	\left( 
	\rho_n \| \omega' \omega^{-1} \|^2_{\infty,n} 
	+ 
	\| \omega' \omega^{-1} \|^4_{\infty,n} 
	\| \omega'' \omega^{-1} \|^2_{\infty,n} 
	+
	\rho_n^2 
	\right)
	= \os(\|A\|_{\infty,n}). \\[-8mm]
	\end{align}
	We also note that \eqref{A.lim} implies that
	\begin{equation}
	\sup_{u,v \in \supp \varphi_n}\left|\log \frac{A(u)}{A(v)} \right| \leq \int_{\supp \varphi_n} \! \! \! \frac{|A'(t)|}{A(t)} \; \dd t
	= \BigO(\rho_n^\frac 12), \quad n \to \infty, 
	\end{equation}
	hence 
	\begin{equation}\label{A.comp}
	\frac{\|A\|_{\infty,n}}{\ds \inf_{u \in \supp \varphi_n}A(u)} = \BigO(1), \quad n \to \infty.
	\end{equation}
	
	Clearly, we have
	\begin{equation}
	\psi_\la'(x_1) = \ii A(x_1)^\frac12 \psi_\la(x_1), 
	\qquad 
	\psi_\la''(x_1) = - A(x_1) \psi_\la(x_1) +\frac{\ii}{2} \frac{A'(x_1)}{A(x_1)^\frac12} \psi_\la(x_1)
	\end{equation}
	and, since $|\psi_\la|=1$,
	\begin{equation}
	\|\nabla \phi_n \| \geq \|\partial_1 \phi_n\| 
	\geq
	\|\varphi_n \psi_\la' \chi\| - \|\varphi_n' \chi\| - \|\varphi_n \partial_1 \chi\|. 
	\end{equation}
	By straightforward calculations and using that $|x'| < \omega(x_1)^{-1}$ for $x \in U_\omega$ as well as \eqref{om.reg} and \eqref{A.comp}, we arrive at
	\begin{equation}
	\begin{aligned}
	\|\varphi_n \psi_\la' \chi\|^{-2} = \BigO(\|A\|_{\infty,n}) , 
	\quad
	\|\varphi_n' \chi\| + \|\varphi_n \partial_1 \chi\| = \os(1), \quad n \to \infty,
	\end{aligned}
	\end{equation}
	whence
	\begin{equation}
	\|\nabla \phi_n \|^{-2} = \BigO(\|A\|_{\infty,n}), \quad n \to \infty.
	\end{equation}
	On the other hand, tedious but straightforward, and hence omitted, calculations and estimates yield that
	\begin{equation}
	\|(-\Delta - A) \phi_n\|^2 = \os(\|A\|_{\infty,n}), \quad n \to \infty.
	\end{equation}

	Finally, assumption \eqref{B.A.lim} implies that $\|B\phi_n\|^2 = \os(\|A\|_{n,\infty})$ as $n \to \infty$ and so \eqref{sing.s.1} follows. 
\end{proof}

\begin{remark}\label{rem:radial}
	If $\Omega$ contains a cone $C_\delta$, \cf~\eqref{cone.def}, with some $\delta >0$ and $a$, $q$ are radial functions (or perturbations
	thereof of the type \eqref{B.A.lim} in $C_\delta$), the above construction of a singular sequence can be adapted accordingly. In this case, for
	$\la \in (-\infty,0)$ we have $\la \in \se{2}(G)$ if there exists a decomposition
	\begin{equation}
	q(x) + 2\la a(x) + \la^2 = \widetilde A(|x|) + B(x), \quad x \in C_\delta,
	\end{equation}
	such that $A(u):= \widetilde A(|x|)$ and $B$ satisfy conditions~\eqref{A.lim} and~\eqref{B.A.lim}.
	
	We mention that, in spherical coordinates $x=(|x|,\Theta) \in (0,\infty) \times S^{d-1}$, where $S^{d-1}$ is the $(d-1)$-dimensional unit sphere,
	a suitable singular sequence has the`form
	\begin{equation}
	\phi_n(|x|,\Theta) := |x|^{-\frac{d-1}{2}} \varphi_n(|x|) \psi_\la(|x|) \chi(\Theta), \quad n\in\N,
	\end{equation}
	\vspace{-1mm}where 
	\begin{alignat*}{2}
	\psi_\la(|x|) & := \exp \left(\ii \int_\alpha^{|x|} \widetilde A(t)^\frac 12 \dd t \right), \quad
	&
	0 &\neq \chi \in C_0^\infty(S^{d-1} \cap C_\delta),
	\\
	\varphi_n(|x|) & := \rho_n^\frac14 \varphi(\rho_n^\frac12 |x|-n), &
	0 &\neq \varphi \in C_0^\infty((0,1)). 
	\end{alignat*}
\end{remark}

\vspace{0.2mm}

\section{Convergence of non-real eigenvalues}
\label{sec:lim.an}

In this section we consider a sequence of dampings $\{a_n\}$ that are unbounded at infinity in the sense of Assumption~\ref{asm:a.inf} and which
converge in a suitable sense to a limit function $a_\infty$ on some open subset $\Omega_\infty \subset \Omega \subset \Rd$. 

To this end, we study the spectral convergence for the quadratic operator functions
\begin{equation}\label{Tn.def}
T_n(\la):=-\Delta + q + 2 \la a_n + \la^2, \quad n \in \N^*:=\N \cup \{\infty\}, \quad \la \in \Cla,
\end{equation}
in $\LOm$ for $n \in \N$ and in $L^2(\Omega_\infty) \subset \LOm$ for $n=\infty$. 
While we allow for the case $\Omega_\infty = \Omega$, the example \eqref{an.ex}, \eqref{ainf.ex} discussed in the introduction illustrates the need to consider dampings $a_n$ that diverge on the non-empty interior of $\Omega \setminus \Omega_\infty$, and hence for $T_n(\la)$ and $T_\infty(\la)$ acting in possibly different spaces $\LOm$ and $L^2(\Omega_\infty)$. In fact, the dampings $a_n$ are only supposed
to converge to $a_\infty$ in $L^2_{\rm loc}(\Omega_\infty)$. Recall that, for $\{b_n\} \subset L^2_{\rm loc}(\Omega')$, $b \in L^2_{\rm loc}(\Omega')$
and $\Omega' \subset \R^d$ open, we have $b_n \to b$ in $L^2_{\rm loc}(\Omega')$ as $n\to \infty$, if for all compact sets $K\subset \Omega'$, 
\begin{equation}
\int_K |b_n-b|^2 \to 0, \quad n \to \infty.
\end{equation}
We shall also need the so-called \emph{segment condition} for $\Omega_\infty$ which means that the domain $\Omega_\infty$ does not lie on both sides of part of its boundary or, more precisely, that  every $x \in \partial \Omega$ has a neighborhood $U_x$ and a non-zero vector $y_x \in \Rd$ such that if $z \in \ov \Omega \cap U_x$, then $z + t y_x \in \Omega$ for $0<t<1$,  \cf~\cite[Sec.~3]{Adams-2003}.

Our convergence result in Theorem~\ref{thm:sp.conv} below is formulated for quadratic operator functions $T_n$, $n \in \N$, 
requiring $a_n$ and $q$ to be only in $L^1_{\rm loc}(\Omega;\R)$. If even Assumption~\ref{asm:a.r} is
satisfied, then spectral convergence for the corresponding generators $G_n$, $n \in \N$, follows from this result by Theorem~\ref{thm:GT}. 

\begin{asm}\label{asm:lim}
Let $\emptyset \neq \Omega_\infty \subset \Omega \subset \Rd$ be open and assume that $\Omega_\infty$ satisfies the segment condition.
Suppose that
\begin{enumerate}[{\upshape (\ref{asm:lim}.i)},align=left,labelwidth=0.35in]
	\item \label{asm:lim.reg}
	$q \in \LolocOmR$, $\{a_n\}_{n \in \N_0} \subset \LolocOmR$ and $a_\infty \in L^1_{\rm loc}(\Omega_\infty;\R)$,
	\item \label{asm:lim.ang}
	for all $n \in \N_0$, $a_n \geq 0$ and
	\begin{equation}
	\lim_{k \to \infty} \ \essinf_{x \in \Omega, |x|>k} \ a_n(x) = \infty,
	\end{equation}

	\item \label{asm:lim.a1} 
	for all $n \in \N$, $a_n \geq a_0$ in $\Omega$ and $a_\infty \geq a_0$ in $\Omega_\infty$, 
	\item \label{asm:lim.Om}
	$a_n^{\frac 12} \restriction \Omega_\infty \to  a_\infty^{\frac 12}$ in $L^2_{\rm loc}(\Omega_\infty)$, $n\to\infty$,
	\item \label{asm:lim.inf}
	for all $n \in \N$, $a_n^{-\frac 12} \restriction \Omega_0 \in L^2_{\rm loc}(\Omega_0)$ and $a_n^{-\frac 12} \to 0$ in $L^2_{\rm loc}(\Omega_0)$, $n\to\infty$,
	where $\Omega_0 := (\Omega \setminus \Omega_\infty)^{\rm o}$.
\end{enumerate}
\end{asm}
Note that Assumption \ref{asm:lim.inf} is relevant only when $(\Omega \setminus \Omega_\infty)^{\rm o} \neq \emptyset$, which is not excluded here. The quadratic operator functions
$T_n(\la)=H_{2\la,n} + \la^2$, $\la \in \Cla$, are defined as in Section \ref{sec:T.def}, via the Schr\"odinger operators 
\begin{equation}\label{Hgn.def}
H_{\gamma,n}:= -\Delta + q + \gamma a_n, \quad n \in \N^*:=\N \cup \{\infty\}, \quad \gamma \in \Cla,
\end{equation}
and Assumption~\ref{asm:lim.ang} ensures that Assumption~\ref{asm:a.gr} is satisfied. Thus the non-real spectrum of $T_n$ consists only of eigenvalues by
Proposition~\ref{prop:sp.T}. 

The main result of this section is the following \emph{spectral exactness} theorem for $\{T_n\}$,  $n\in\N$. The latter
means that all eigenvalues of the limiting operator function
$T_\infty$ are approximated by eigenvalues of $T_n$ and all finite accumulation points of eigenvalues of $T_n$ outside $(-\infty,0]$ are eigenvalues of
$T_\infty$, \ie~no spectral pollution occurs. An illustration of this result may be found in example~\eqref{an.ex},~\eqref{ainf.ex} in
Section~\ref{subsec:x2n}.
\begin{theorem}
\label{thm:sp.conv}
	
Let Assumption \ref{asm:lim} be satisfied and let $\{T_n(\la)\}_{n \in \N^*}$, $\la \in \Cla$, be as in \eqref{Tn.def}.
Then the following hold.

\begin{enumerate}[\upshape i)]
	\item If $\la  \in \spp(T_{\infty})$, then there exists a sequence $\{\lambda_n\}_{n\in \N}$, $\la_n \in \spp(T_n)$, such that $\lambda_n \to \lambda,$ $n \to \infty$.
	Conversely, if $\{\lambda_n\}_{n\in \N}$, $\la_n \in \spp(T_n) \subset \Cla$, has a subsequence $\{\lambda_{n_k}\}_{k\in \N}$ such that $\lambda_{n_k} \to \lambda \in \C \setminus (-\infty,0]$, $k \to \infty$, then $\lambda \in \spp(T_\infty)$.
	\item If $\lambda_n \to \lambda$, $n\to \infty$, where $\lambda_n \in \spp(T_n)$, $\lambda \in \spp(T_\infty)$ and $\{f_n\}$ is a sequence of normalized eigenfunctions of $T_n$ at $\lambda_n$, then the sequence $\{f_n\}$ is compact in $\LOm$ and its accumulation points $($which belong to $L^2(\Omega_\infty))$ are normalized eigenvectors of $T_\infty$ at $\lambda$.
\end{enumerate}
\end{theorem}

In the first step of the proof of Theorem~\ref{thm:sp.conv}, we establish generalized strong resolvent convergence of $H_{\gamma,n}$ to $H_{\gamma,\infty}$ as $n\to\infty$ for all $\gamma =2\la\in \C \setminus (-\infty,0]$, \cf~Proposition~\ref{prop:Hn.nrc}; here ``generalized'' refers to the fact that the operators act in possibly different spaces; this is reflected by the presence of the characteristic function $\chi_\infty$ of $\Omega_\infty$ in \eqref{res.conv.n.1} below. 
In the second step, we employ abstract spectral convergence results for analytic Fredholm operator functions \cite[Satz~4.1.(18)]{Vainikko-1976}
for $T_n(\la)=H_{2\la,n} + \la^2$, $\la \in \Cla$, $n \in \N$.

\begin{proposition}\label{prop:Hn.nrc}
Let Assumption \ref{asm:lim} be satisfied and let $\{H_{\gamma,n}\}_{n \in \N^*}$, $\gamma \in \Cla$ be as in \eqref{Hgn.def}. 
Then, for all $\gamma \in \C \setminus (-\infty,0]$ and every $f \in \LOm$, 
\begin{equation}\label{res.conv.n.1}
\|(H_{\gamma,n} + I)^{-1}f - (H_{\gamma,\infty} + I)^{-1} \chi_\infty f \|  \to 0, \quad n \to \infty,
\end{equation}
where $\chi_\infty$ is the characteristic function of $\Omega_\infty$.
\end{proposition} 
\begin{proof}
To simplify the notation within this proof, we drop the subscript $\gamma$ in the sequel and denote $\omega:=\arg(\gamma)/2$. 
First we notice that
\begin{equation}
-1 \in {\ds \bigcap_{n \in \N^*}} \rho(H_n)
\end{equation}
by the numerical range enclosure, \cf~\eqref{h.gam.nr}, and the fact that $\widetilde H_n = \e^{-\ii \omega} H_n$, $n \in \N^*$, is \mbox{m-sectorial}, \cf~\eqref{H.gam.def.2}.
Clearly, \eqref{res.conv.n.1} is equivalent to
\begin{equation}
\|(\widetilde H_n + \e^{-\ii \omega})^{-1} f - (\widetilde H_\infty + \e^{-\ii \omega})^{-1} \chi_\infty f\| \to 0, \quad n \to \infty,
\end{equation}
which we prove by contradiction in the following. 

Suppose that there exists a function $f$ in $\LOm$ and a $\delta>0$ such that 
\begin{equation}\label{Hn.contr}
\|(\widetilde H_n + \e^{-\ii \omega})^{-1} f - (\widetilde H_\infty + \e^{-\ii \omega})^{-1} \chi_\infty f\| \geq \delta >0
\end{equation}
for all $n \in J$ for some infinite subset $J \subset \N$. Then, for $\psi_n := (\widetilde{H}_n + \e^{-\ii \omega})^{-1}f$, $n \in J$, 
\begin{equation}\label{hn.psin}
\widetilde{h}_n [\psi_n] + \e^{-\ii \omega} \|\psi_n\|^2 = \langle f, \psi_n \rangle, \quad  n \in J,
\end{equation}
and the enclosure of the numerical range~\eqref{h.gam.nr} implies that
\begin{align}\label{psin.F}
\hspace{-2mm}
\|\psi_n \| \leq \|(\widetilde H_n + \e^{-\ii \omega})^{-1}\|\|f\| \leq \frac{\|f\|}{\dist(-\e^{-\ii \omega},\Num(\widetilde h_n))} \leq \frac{\|f\|}{\cos \omega},
\quad  n \in J.
\end{align}
Taking real parts in~\eqref{hn.psin} and using~\eqref{psin.F}, we get
\begin{equation}\label{h.n.est}
\|\nabla \psi_n\|^2 + \| q^\frac12 \psi_n\|^2 + |\gamma| \| a_n^\frac12 \psi_n\|^2 + \|\psi_n\|^2
\leq  
\frac{\|f\|^2}{(\cos \omega)^2}, \quad n \in J. 
\end{equation}
This shows that $\{\psi_n\}_{n \in J}$ is bounded in the Hilbert space $(\H_0,\langle \cdot, \cdot\rangle_{\H_0})$ defined by
\begin{equation}\label{H1.def}
\begin{aligned}
\cH_0 &:= \WotOmD \cap  \Dom(q^\frac12) \cap  \Dom(a_0^\frac12), 
\\
\langle \cdot, \cdot\rangle_{\H_0} &:= \langle  \cdot,  \cdot \rangle_{W^{1,2}}  
+ 
\langle  q^\frac12 \cdot, q^\frac12 \cdot \rangle
+
\langle a_0^\frac12 \cdot, a_0^\frac12 \cdot \rangle.
\end{aligned}
\end{equation}
Thus $\{\psi_n\}_{n \in J}$ has a weakly convergent subsequence $\{\psi_n\}_{n \in J'}$ where $J'$ is an infinite subset of $J$, 
in $\H_0$. Since the embedding
$\H_0 \hookrightarrow \LOm$ is compact due to Assumption~\ref{asm:lim.a1}, \cf~the proof of Lemma~\ref{lem:h.def}.\ref{H.comp.res}, $\{\psi_n\}_{n \in J'}$
converges in $\LOm$.
Moreover, \eqref{h.n.est} shows that $\{a_n^{1/2} \psi_n\}_{n \in J}$ is bounded in $\LOm$, thus we can assume that $\{a_n^{1/2}\psi_n\}_{n \in J'}$ converges
weakly in $\LOm$. Altogether, there exist $\psi \in \H_0$ and $g \in \LOm$ such that, for $n \in J'$ and as $n \to \infty$,
\begin{align}
&\forall \, \eta \in \H_0 \ \ \langle \psi_n, \eta \rangle_{\H_0} \lto  \langle \psi, \eta \rangle_{\H_0},  
\label{psink.H1}
\\
&\|\psi_n - \psi\| \lto 0, \label{psink.str}
\\
&\forall \, \zeta \in \LOm \ \ \langle a_n^\frac12 \psi_n, \zeta \rangle \lto  \langle g, \zeta \rangle.    \label{psink.G}
\end{align}

If $\Omega_0 =(\Omega\setminus\Omega_\infty)^{\rm o} \neq \emptyset$, we choose arbitrary $\varphi \in C_0^{\infty}(\Omega_0)$. Then, by Assumption \ref{asm:lim.inf} and the boundedness of
$\| a_n^{1/2} \psi_n \|$, \cf~\eqref{h.n.est}, we obtain
\begin{equation}
|\langle \psi_n,\varphi \rangle |
= 
|\langle a_n^{\frac12} \psi_n,a_n^{-\frac12} \varphi \rangle| 
\leq 
\| a_n^{\frac12} \psi_n \| \| a_n^{-\frac12} \varphi  \| \lto 0, \quad n \in J', \ n \to \infty.
\end{equation}
Hence $\langle \psi,\varphi \rangle  = {\ds\lim_{n \in J', n \to \infty}} \langle \psi_n,\varphi \rangle  = 0$,
and so $\psi=0$ a.e.~in $\Omega_0$. Since $\psi \in \H_0 \subset \WotOmD$, the latter implies that
$\psi \restriction \Omega_\infty \in W^{1,2}_0(\Omega_\infty)$, \cf~\cite[Lem.~3.27, Thm.~5.29]{Adams-2003}.

Now let $\phi \in C_0^\infty(\Omega_\infty)$. By Assumption \ref{asm:lim.Om}, $a_n^{1/2} \phi \rightarrow a_\infty^{1/2} \phi$ in $L^2(\Omega_\infty)$ as $n\to\infty$ and thus
$\sup_{n \in J'} \|a_n^{1/2} \phi\|<\infty$. Hence \eqref{psink.str} implies that
\begin{equation}\label{an.eq.1}
\langle a_n^\frac12 \psi_n,\phi  \rangle
=
\langle \psi_n -\psi, a_n^\frac12 \phi  \rangle + 
\langle \psi, a_n^\frac12 \phi  \rangle 
\lto 
\langle \psi, a_\infty^\frac12 \phi  \rangle, 
\quad n \in J', \ n \to \infty.
\end{equation}
On the other hand,  $\langle a_n^{1/2} \psi_n, \phi  \rangle \to \langle g, \phi \rangle$, \cf~\eqref{psink.G}. Since $\phi \in C_0^\infty(\Omega_\infty)$ was
arbitrary, $g \restriction \Omega_\infty=a_\infty^{1/2} \psi \restriction \Omega_\infty$ a.e. in $\Omega_\infty$. Therefore
$\{a_n^{1/2}\psi_n  \restriction \Omega_\infty\}_{n \in J'}$ converges weakly to $a_\infty^{1/2} \psi\restriction \Omega_\infty$ in $L^2(\Omega_\infty)$.
Using $\sup_{n \in J'} \|a_n^{1/2} \psi_n\| < \infty$, \cf~\eqref{h.n.est}, and Assumption~\ref{asm:lim.Om}, we finally obtain, for $n \in J'$
and as $n \to \infty$, 
\begin{equation}\label{psinan.conv}
\langle a_n^\frac12 \psi_n, a_n^\frac12 \phi  \rangle
=
\langle a_n^\frac12 \psi_n, (a_n^\frac12 - a_\infty^\frac12) \phi \rangle
+
\langle a_n^\frac12 \psi_n, a_\infty^\frac12 \phi  \rangle
\lto \langle a_\infty^\frac12 \psi, a_\infty^\frac12 \phi  \rangle.
\end{equation}

In summary, \eqref{psink.H1}, \eqref{psink.str} and \eqref{psinan.conv} show that, for any $\phi \in C_0^{\infty}(\Omega_\infty)$ and for $n \in J'$, $n \to \infty$,
\begin{equation}\label{h.n.lim}
\begin{aligned}
\langle f, \phi \rangle_{L^2(\Omega_\infty)}  =  \langle f, \phi \rangle  
&= \widetilde{h}_n(\psi_n, \phi) + \e^{-\ii \omega} \langle \psi_n, \phi \rangle 
\\
& \to \e^{-\ii \omega}
\left(\langle \nabla \psi, \nabla \phi \rangle_{L^2(\Omega_\infty)} 
+ 
\langle q^\frac12 \psi, q^\frac12 \phi \rangle_{L^2(\Omega_\infty)} 
\right)
\\
& \quad +
\e^{\ii \omega}  |\gamma| \langle a_\infty^\frac12 \psi, a_\infty^\frac12 \phi  \rangle_{L^2(\Omega_\infty)} 
+
\e^{-\ii \omega} \langle \psi, \phi \rangle_{L^2(\Omega_\infty)}
\\
&=\widetilde h_\infty(\psi,\phi) + \e^{-\ii \omega} \langle \psi, \phi \rangle_{L^2(\Omega_\infty)},
\end{aligned}
\end{equation}
and hence the first and the last term in~\eqref{h.n.lim} must be equal. This and the representation theorem \cite[Thm.~VI.2.1]{Kato-1966} imply that
$\psi  \restriction \Omega_\infty \in \Dom(\widetilde{H}_{\infty})$ and 
$(\widetilde{H}_{\infty} + \e^{-\ii \omega}) (\psi \restriction \Omega_\infty) = f \restriction \Omega_\infty$,
\ie~$(\widetilde{H}_{\infty} + \e^{-\ii \omega})^{-1} (f \restriction \Omega_\infty)   = \psi \restriction \Omega_\infty$. 
The latter and \eqref{psink.str} yield that $\|(H_n + \e^{-\ii \omega})^{-1} f - (H_\infty + \e^{-\ii \omega})^{-1} \chi_\infty f \| \to 0$
as $n \to \infty$ with $n \in J' \subset J$, a contradiction to \eqref{Hn.contr}.
\end{proof}  
\begin{proof}[Proof of Theorem~\upshape \ref{thm:sp.conv}]
	
	We define operator functions $A_n$ and $A_\infty$ whose values are bounded linear operators in $\LOm$ and $L^2(\Omega_\infty)$, respectively, by
	\begin{equation}\label{A.def}
	A_n(\lambda):= (H_{2\la,n}+\la^2)(H_{2\la,n}+I)^{-1}= I + (\la^2-1) \left(H_{2\lambda,n} + I \right)^{-1},
	\end{equation}
	$n \in \N^*$, $\la \in \Cla$; these functions are well-defined since $-1 \in {\ds\bigcap_{n \in \N^*}} \rho(H_{2\la,n})$ due to \eqref{h.gam.nr}.
	
	It is easy to verify that, for every $n \in \N^*$, if a nonzero $\psi$ in $\Dom(H_{2\la,n})$ satisfies $(H_{2\lambda,n} + \lambda^2)\psi=0$, then $A_n(\lambda)\psi=0$ and, conversely, if a nonzero $\psi$ in $\LOm$ or in $L^2(\Omega_\infty)$ for
	$n\!=\!\infty$ satisfies $A_n(\la) \psi \!=\!0$, then $\psi \!\in\! \Dom(H_{2\la,n})$ and $(H_{2\lambda,n} + \lambda^2)\psi\!=\!0$. 
	Hence, for $\la \in \Cla$ and  $n \in \N^*$, 
	\begin{equation}
	0 \in \sigma_{\rm p} (H_{2\lambda,n} + \lambda^2) 
	\iff
	0 \in \sigma_{\rm p} (A_n(\lambda)). 
	\end{equation}
	The claims of Theorem~\ref{thm:sp.conv} will follow from convergence results for holomorphic operator functions, 
	\cf~\cite[Satz~4.1.(18)]{Vainikko-1976} or the summary in \cite[Sec.~1.1.2]{Boegli-2014-PhD}, if we verify the following assumptions therein:
	\begin{enumerate}[a)]
		\item \label{An.hol} $\la \mapsto A_n(\lambda)$, $n \in \N^*,$ is holomorphic in $\Cla$,
		\item \label{An.Fred} $A_n(\lambda)$, $n \in \N^*$, $\la \in \Cla$, is Fredholm with index $0$,
		\item \label{An.reg} for all $\la \in \Cla$, $\{A_n(\la)\}$ converges regularly to $A_\infty(\la)$, 
		\item \label{An.reg.p} there exists $\la_0 \in \Cla$ such that $\la_0 \in \rho(A_\infty)$,
		\item \label{An.K} for every $K \subset \Cla$, $K$ compact, 
	\end{enumerate}
	\begin{equation}\label{An.sup.K}
	\sup_{n \in \N} \ \max_{\lambda \in K} \|A_n(\lambda)\| < \infty.
	\end{equation}

	We have already shown that~\ref{An.hol} holds for $n\in \N^*$, \cf~Lemma~\ref{lem:h.def} \ref{H.hol.it}.
	The validity of condition \ref{An.Fred} follows from \cite[Thm.~IX.2.1]{EE} since $A_n(\lambda)-I$, $n \in \N^*$, is a compact
	operator,~\cf~Lemma~\ref{lem:h.def}.\ref{H.comp.res}. For \ref{An.reg.p}, we observe that $A_\infty(1) = I$ so we can choose $\la_0=1$. The bound in \ref{An.K}
	follows immediately from the compactness of $K$ and from
	\begin{equation}\label{Hn.res.bd}
	\|(H_{2\la,n}+I)^{-1}\| 
	\leq
	\dist(-\e^{-\ii \frac{\arg(\la)}{2}}, \Num(\widetilde h_n))^{-1}
	\leq \frac{1}{\cos \frac{\arg(\la)}{2}}, \quad n \in \N^*,
	\end{equation}
	\cf~\eqref{h.gam.nr} and \eqref{H.gam.def.2}. 
	
	The only remaining point is \ref{An.reg}, the \emph{regular convergence}, \cf~\cite{Vainikko-1976} or the summary in \cite[Sec.~1.1.2]{Boegli-2014-PhD}. In detail, we need to show that 
	
	i) for any $\{\psi_n\} \subset \LOm$ and $\psi \in L^2(\Omega_\infty)$ such that $\|\psi-\psi_n\| \to 0$, $n \to \infty$, we have
	$\|A_n(\lambda) \psi_n - A_{\infty}(\lambda) \psi\| \to 0$, $n \to \infty$, and 
	
	ii) for any bounded $\{\psi_n\} \subset \LOm$ such that every infinite
	subsequence of $\{A_n(\lambda) \psi_n\}$ contains a convergent subsequence, every infinite subsequence of  $\{\psi_n\}$ also contains a convergent subsequence.
	
	The validity of condition i) follows from \eqref{Hn.res.bd} and the resolvent convergence proved in  Proposition~\ref{prop:Hn.nrc}. In fact, 
	since $\psi=\chi_\infty \psi$ where $\chi_\infty$ is the characteristic function of $\Omega_\infty$,
	\begin{equation}\label{An.conv.1}
	\begin{aligned}
	&\|A_n(\lambda) \psi_n - A(\lambda) \psi\|
	\\
	&\leq \|\psi_n - \psi\|
	+ (|\la|^2+1) \|(H_{2\lambda,n} + I )^{-1} \psi_n - (H_{2\lambda,\infty} + I )^{-1} \psi\|
	\\
	& \leq 
	\left(1 + (|\la|^2+1) \|(H_{2\lambda,n} + I )^{-1}\| \right)  \|\psi_n - \psi\| 
	\\
	&
	\quad + (|\la|^2+1) \|(H_{2\lambda,n} + I )^{-1} \psi - (H_{2\lambda,\infty} + I )^{-1} \psi\| \lto 0, \quad n \to \infty.
	\end{aligned}
	\end{equation}

	To verify condition ii), due to the relations,
	\begin{equation}
	\psi_n  = A_n(\lambda) \psi_n - (\la^2-1) (H_{2\lambda,n} + I )^{-1} \psi_n, \quad n \in \N,
	\end{equation}
	it suffices to show that $\{(H_{2\lambda,n} + I )^{-1} \psi_n \}_{n \in J}$ with infinite $J \subset \N$ has a convergent subsequence. This can be
	shown in a similar way as in~\eqref{psin.F}--\eqref{H1.def}. In detail, $\{(H_{2\lambda,n} + I )^{-1} \psi_n \}$, and hence
	$\{\widetilde H_{2\la,n} + \e^{-\ii \frac{\arg (\la)}{2}}\}$, is bounded due to the boundedness of $\{\psi_n\}$ and \eqref{Hn.res.bd}. Thus, as
	in~\eqref{hn.psin}, we obtain
	\begin{equation}
	\widetilde h_{2\la,n}[\psi_n] + \e^{-\ii \frac{\arg (\la)}{2}} \|\psi_n\|^2 = \langle (\widetilde H_{2\la,n} + \e^{-\ii \frac{\arg (\la)}{2}}) \psi_n, \psi_n \rangle, \quad n\in\N.
	\end{equation}
	and we proceed as in the paragraphs below~\eqref{H1.def} to finish the proof of ii) and hence of the theorem.
\end{proof}

\section{Examples}
\label{sec:ex}

As an illustration of our abstract results, we fully characterize the spectrum of the generator $G$ 
for several examples of damping terms. For $\la \in (-\infty,0)$ and for $\la=0$, 
we employ the result on the essential spectrum of 
\begin{equation}
G=\begin{pmatrix}
0 & I \\
\Delta - q & - 2 a  
\end{pmatrix},
\end{equation}	
\cf~Theorem~\ref{thm:Gspe}, while for $\la \in \Cla$ we use the spectral correspondence $\la \in \sigma(G) \Longleftrightarrow 0 \in \sigma(T(\la))$
between $G$ and the quadratic operator function  
\begin{equation}
T(\la) = -\Delta + q + 2\la a + \la^2, \qquad \la \in \Cla,
\end{equation}
\cf~Theorem~\ref{thm:GT}. 
As mentioned in the introduction, these examples show that the growing damping term will prevent uniform exponential decay of solutions by the creation of essential spectrum covering the entire negative semiaxis $(-\infty,0]$, independently of the existence of eigenvalues with real parts converging to $0$. That this effect is quite general and not restricted to some particular examples may be seen from our abstract results, see Theorem~\ref{thm:Gspe}.

\subsection{Examples for $d=1$}
\label{subsec:x2n}

We start with the family of examples~\eqref{an.ex} with $\Omega= \R$ described in the introduction where the dampings are given by
\begin{equation}
a_n(x)=x^{2n}+\fra_0, \quad x \in \Omega =\R, \quad n \in \N, \quad \fra_0 \geq 0,
\end{equation}
and we consider a constant potential $q(x) \equiv \frq_0 \geq 0$. 
The non-real eigenvalues of the corresponding generators $G_n$ can be expressed in terms of the eigenvalues $\{\mu_k(n)\}_{k \in \N_0} \subset (0,\infty)$ of
the self-adjoint anharmonic oscillators 
\begin{equation}\label{anho.def}
S_n=\Dt + x^{2n}, \quad \Dom(S_n)=W^{2,2}(\R) \cap \Dom(x^{2n}),
\end{equation}
in $L^2(\R)$. These eigenvalues are known to satisfy
\begin{equation}\label{muk.def}
\mu_k(n) = 
\begin{cases}
2k+1, \quad  k \in \N_0, & n=1,
\\[2mm]
\ds \left(\frac{\pi}{\Sigma_{2n}} \right)^{\frac{2n}{n+1}} k^{\frac{2n}{n+1}}\left(1+  \os_k(1)\right), \quad k \to \infty, & n \geq 2,
\end{cases}
\end{equation}
where $\Sigma_{2n} := \int_{-1}^1 (1-x^{2n})^\frac 12 \, \dd x$, \cf~for instance~\cite{Titchmarsh-1954-5}.

\begin{proposition}\label{prop:x2n.ex}
	Let $\Omega = \R$, let $G_n$ be as in \eqref{G.def} with $q(x) \equiv \frq_0 \geq 0$ and $a_n(x)=x^{2n}+\fra_0$, $x \in \R$, $n \in \N$, $\fra_0 \geq0$,
	and let $\{\mu_k(n)\}_{k \in \N_0}$ be the eigenvalues of $S_n$ defined by~\eqref{anho.def}. 
	Then
	\begin{equation}
	\sigma(G_n) = (-\infty, 0] \, \dot\cup \! \bigcup_{k \in \N_0} \! \left\{
	\lambda_k{(n,\fra_0,\frq_0)},\ov{\lambda_k{(n,\fra_0,\frq_0)}} 
	\right\}  \subset \{ \la \in \C \, : \, \Re \la \le 0 \}, \ \ n \in \N,
	\end{equation}
	where $\la_k(n,\fra_0,\frq_0)$, $k \in \N_0$, are the solutions of 
	\begin{equation}
	(\la^2+2 \la \fra_0 + \frq_0)^{n+1} = 2\la (-\mu_k(n))^{n+1}, \quad \Re \la \leq 0, \quad \Im \la >0.
	\end{equation}
	Moreover, all non-real eigenvalues satisfy
        \[
            \sigma(G_n) \setminus \R \subset \{ \la \in \C \, : \, \Re \la \le - \fra_0, \, |\la| \ge \frq_0 \},
        \]
	and, for any $n \in \N$, as $k \to \infty$,
	\begin{equation}\label{la_k.n.a0}
	\lambda_k{(n,\fra_0,\frq_0)} = 2^{\frac{1}{2n +1}} \e^{\ii \frac{n+1}{2n +1} \pi} \big[\mu_k{(n)} \big]^{\frac{n+1}{2n +1}} -
	\frac{2(n+1)}{2n+1}\fra_0(1+ \os_k(1));
	\end{equation}
	in particular, for $n=1$ and $\fra_0=\frq_0=0$, 
	\begin{equation}
	\lambda_k{(n,0,0)}  = 2^\frac 13 \e^{\ii \frac 23 \pi} (2k+1)^\frac 23, \quad k \in \N_0.
	\end{equation}
\end{proposition}
\begin{proof}
	Both the damping $a_n$ and the potential $q$ clearly satisfy Assumption~\ref{asm:a.r}.
	
	That $(-\infty,0) \subset \se{2}(G_n)$ follows from Theorem~\ref{thm:Gspe}  since $a_n'(x)/a_n(x) = \BigO(1/x)$ and $q(x) = \os(a_n(x))$ as
	$x \to + \infty$; because the spectrum is closed, we obtain $(-\infty,0] \subset \sigma(G_n)$.
	
	Since $a_n$ is unbounded at infinity and hence satisfies Assumption~\ref{asm:a.inf}, Theorem~\ref{thm:GT} and 
	Proposition~\ref{prop:sp.T}~\ref{prop:sp.T.i} imply that
	$\sigma(G_n) \setminus (-\infty,0]$ consists only of 
	complex conjugate pairs of eigenvalues $\lambda_k{(n,\fra_0,\frq_0)}$, $k\in\N_0$, 
	in the closed left half plane of finite multiplicity which satisfy $0\in\sigma_{\rm disc}(T_n(\la))$.
	Thus, it suffices to consider $\la \in \C$ with $\Re \la \leq 0 $ and $\Im \la >0$ and we search for solutions $y \in \Dom(T_n(\la))$, $ y \neq 0$, of	
\begin{equation}\label{y.eq}
-y''(x) + 2 \lambda x^{2n} y(x) = - (\lambda^2 + 2 \la \fra_0 + \frq_0) y(x), \quad  x \in \R
\end{equation}
The (complex) change of variable 
\begin{equation}
x = 2^{-\frac{1}{2n+2}} \lambda^{-\frac{1}{2n+2}} z, \quad x \in \R,
\end{equation}
leads to the equation
\begin{equation}\label{w.eq}
-w''(z) + z^{2n} w(z) = \mu w(z), \quad z \in \e^{\ii \frac{\arg(\lambda)}{2n+2}} \R, 
\end{equation}
where 
\begin{equation}\label{mu.la.rel}
2 (-\mu)^{n+1} \lambda = 
\left(\lambda^2 + 2 \lambda \fra_0 + \frq_0
\right)^{n+1}.
\end{equation}
Equation~\eqref{w.eq} with complex $z$ was studied  extensively in~\cite{Sibuya-1975}. It is known that every solution of \eqref{w.eq} either decays or blows up exponentially in each Stokes sector
\begin{equation}
S_k:=\left\{ z \in \C: \left|\arg z - \frac{k \pi}{n+1}\right| < \frac{\pi}{2n+2} \right\}, \ k \in \Z.
\end{equation}
Therefore, for \eqref{y.eq} to have a solution $y \in \Dom(T_n(\la))$, it is necessary that $\eqref{w.eq}$ has a solution decaying both in $S_0$ and $S_{n+1}$.
Thus, in fact it suffices to search for decaying solution of $\eqref{w.eq}$ for real $z$, \ie~to investigate the eigenvalues of~\eqref{anho.def},
and, after several manipulations, \eqref{mu.la.rel} yields the asymptotic formula \eqref{la_k.n.a0}.
\end{proof}
As an illustration of Proposition~\ref{prop:x2n.ex}, the spectrum of $G_1$ with $\frq_0 =0$, and $\fra_0=0$ and $\fra_0=3$ is plotted in Figure \ref{fig.x2}.
\begin{figure}[htb!]
\includegraphics[width=0.9 \textwidth]{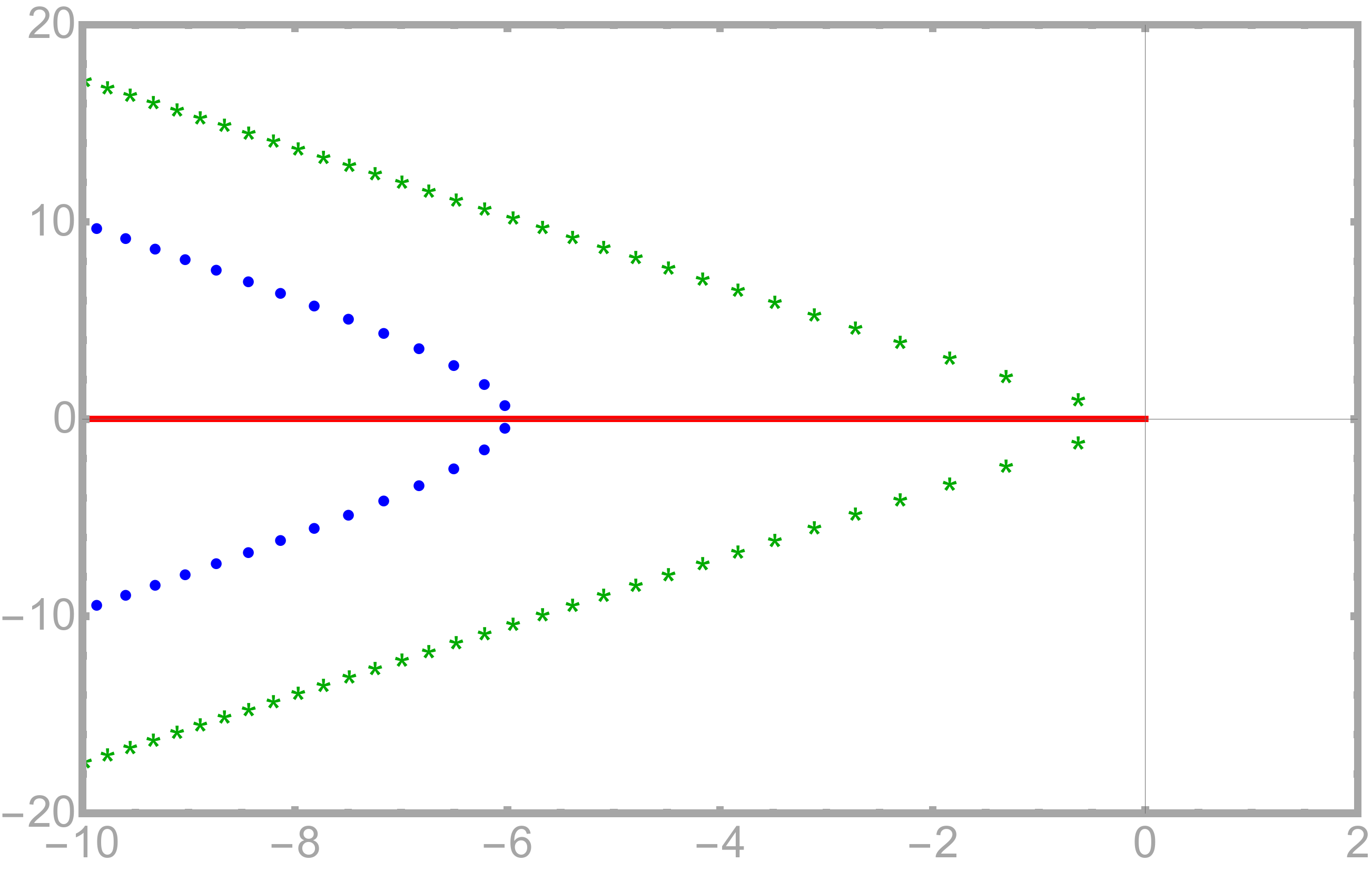}
\caption{Spectrum of $G_1$ with $a_1(x) \!=\! x^2$ and $q \!\equiv\! 0$ in $\Omega \!=\! \R$, \cf~Propo\-sition \ref{prop:x2n.ex} 
with $\fra_0\!=\!0$ (stars/green), $\fra_0\!=\!3$ (dots/blue); the essential spectrum is the semiaxis $(-\infty,0]$ (bold/red) in both cases.}
\label{fig.x2}
\end{figure}

\begin{remark}\label{rem:x.inf}
	Proposition~\ref{prop:x2n.ex} illustrates the convergence of eigenvalues
	\begin{equation}
	\la_k(n,\fra_0,\frq_0) \lto \la_k(\infty,\fra_0,\frq_0), \quad  n \to \infty,
	\end{equation}	
	proved in Theorem~\ref{thm:sp.conv}, where $\{\la_k(\infty,\fra_0,\frq_0)\}$ 
	are the solutions of $0\in \sigma_{\rm disc}(T_\infty(\la))$
	with $\Omega_\infty=(-1,1)$ and $a_\infty=a_0$, \ie~
	\begin{equation}
	T_\infty(\la) = \Dt + \frq_0 + 2 \la \fra_0 + \la^2, \quad \la \in \C,
	\end{equation}
	with Dirichlet boundary conditions at $\pm 1$. 
	It is not difficult to see that the solutions of $0\in \sigma_{\rm disc}(T_\infty(\la))$
	are given by
	\begin{equation}\label{a.inf.EV}
	\lambda_k(\infty,\fra_0,\frq_0) = -\fra_0 \pm \ii \, \sqrt{\mu_k + \frq_0} \, \sqrt{1- \frac{\fra_0^2}{\mu_k + \frq_0}}, 
	\quad \mu_k = \left(\frac{\pi k}{2} \right)^2, \quad   k \in \N.
	\end{equation}
	Note also that we cannot expect uniform convergence in $k$ since, for $\fra_0=0$, the eigenvalues of $T_n$ lie on the two rays
	$\e^{ \pm \ii \frac{n+1}{2n +1} \pi} \R_+$, while the eigenvalues of $T_\infty$, with the possible exception of finitely many,
	lie on the vertical line with $\Re \la = -\fra_0$.
\end{remark}

\subsection{Examples for $d>1$}
\label{subsec:strip}

In higher dimensions analogous examples with
\begin{equation}
a(x) = \sum_{j=1}^d x_j^{2n} + \fra_0, \quad x \in \Omega=\Rd,  \quad \fra_0 \geq 0,
\end{equation}
and a constant potential $q(x) =\frq_0 \geq 0$, $x \in \Rd$, can be analyzed. In particular, the case with $\frq_0=0$ and $d \geq 3$ fits into the assumptions considered in \cite{Ikehata-2017} where a polynomial estimate for the energy decay of the solution was established. In fact,  exponential energy decay cannot occur as Theorem~\ref{thm:Gspe} shows that the essential spectrum of the corresponding operators $G_n$ covers $(-\infty,0]$ for each $n \in \N$; note that it is easy to see that the conditions of Theorem~\ref{thm:Gspe} are satisfied on $U_1$, which is a tube if $d=3$, \cf~\eqref{U.om.def}. The whole non-real part of the spectrum of $G_n$ consists of eigenvalues satisfying 
\begin{equation}
\spp(G_n) \subset \{ \la \in \Cla \,:\, \Re \la \leq - \fra_0, \ |\la|^2 \geq \frq_0 \},
\end{equation}	
\cf~Proposition~\ref{prop:sp.T}.
In fact, separation of variables yields that there are non-real eigenvalues asymptotically approaching rays as for $d=1$, \cf~Proposition~\ref{prop:x2n.ex}, except that now the
corresponding multiplicities depend on the dimension $d$.

For $d=2$ an example with interesting spectrum is obtained for the damped wave equation on a strip of the form
\begin{equation}\label{strip.def}
\Omega = \R \times (-\ell,\ell), \quad \ell >0,
\end{equation}
and with damping $a$ unbounded along the longitudinal direction corresponding to the first variable $x$. As a particular example, we 
consider $a(x,y) = x^2 +\fra_0$ and $q(x,y) = \frq_0 \geq 0$, $(x,y) \in \Omega$. Notice that the associated quadratic operator function
can be viewed as the limit of $T_n(\la)= -\Delta + \frq_0 + 2\la (x^2+ \fra_0 + y^{2n}) + \la^2$ acting in $L^2(\R^2)$ as $n \to \infty$ 
in the sense of Theorem~\ref{thm:sp.conv}.

\begin{proposition}\label{prop:T.strip}
Let $\Omega=\R\times(-\ell,\ell)$ with $\ell>0$, let $G$ be as in \eqref{G.def} with $q(x,y)=\frq_0 \geq 0$ and $a(x,y) = x^2 + \fra_0$, $(x,y)\in\Omega$, where $\fra_0 \geq 0$. 
Then
\begin{equation}
\sigma(G) = (-\infty, 0] \  \dot\cup \!\!\!\! \bigcup_{j \in \N, k \in \N_0} \!\!\! \left\{
\lambda_{jk}{(\fra_0,\frq_0)},\lambda_{jk}{(\fra_0,\frq_0)} 
\right\} \subset \{ \la \in \C \, : \, \Re \la \le 0 \}, 
\end{equation}
where $\la_{jk}(\fra_0,\frq_0)$, $j \in \N$, $k \in \N_0$, are the solutions of 
\begin{equation}\label{strip.ev}
2 \lambda (2k+1)^2 = 
\left[\lambda^2 +\left(\frac{j\pi}{2\ell}\right)^2 + 2\lambda \fra_0 + \frq_0 \right]^2, 
\quad \Re \la \leq 0, \quad \Im \la >0.
\end{equation}
Moreover, all non-real eigenvalues satisfy
\[
  \sigma(G) \setminus \R \subset \{ \la \in \C \, : \, \Re \la \le - \fra_0, \, |\la| \ge \frq_0 \},
\]
and, for fixed $k \in \N_0$, each sequence of eigenvalues $\{\lambda_{jk}(\fra_0,\frq_0)\}_j$ satisfies  
\[
\Re\la_{jk}(\fra_0,\frq_0) \, \nearrow -\fra_{0}, \quad j \to +\infty.
\]
\end{proposition}
\begin{proof}
It is easy to see that $a$ and $q$ satisfy Assumption~\ref{asm:a.r} and~\ref{asm:a.inf}, and hence Theorem~\ref{thm:GT} implies that $\sigma(G) \setminus (-\infty,0]$ consists only of eigenvalues of finite multiplicity 
which may only accumulate at $(-\infty,0]$.

It also not difficult to check that the assumptions of Theorem~\ref{thm:Gspe} are satisfied on a strip $U_{\omega}$ with $\omega(x) = \ell^{-1}$, \cf~\eqref{U.om.def}, \and $B \equiv 0$, which yields
$(-\infty, 0) \subset \se{2}(G)$ and hence also $0 \in \sigma(G)$ since the spectrum is closed.

Since $\essinf{a}=\fra_0$, the first statement on the localization of the real parts of eigenvalues follows from Proposition~\ref{prop:sp.T}.\ref{prop:sp.T.ii}.

The more detailed properties of the eigenvalue sequences $\{\lambda_{jk}(\fra_0,\frq_0)\}_j$ do not follow from our abstract results. They will be obtained using the associated quadratic operator function $T(\la)$ and separation of variables, \ie~by searching for eigenfunctions of the form
$\psi(x,y):=f(x)g(y)$, $(x,y)\in\Omega$. 
The spectral problem in the $y$-variable reduces to the problem
\begin{equation}\label{Dir.EV}
\begin{aligned}
-g''(y)  = \sigma g(y), \qquad 
g(\pm \ell)  =0,
\end{aligned}
\end{equation}
which has the Dirichlet eigenvalues $\sigma_j=(j\pi /(2\ell))^2,$ $j \in \N$. In the $x$-variable, we are left with a family of spectral problems in
$L^2(\R)$ for $T_1(\la) + \sigma_j$ where $T_1(\la) = -\dd^2/\dd x^2 + \frq_0 + 2\la (x^2+\fra_0) + \la^2$, $\la\in\Cla$, is the quadratic operator function analyzed in
Proposition~\ref{prop:x2n.ex}. The change of variables and Stokes sectors argument as in the proof of Proposition~\ref{prop:x2n.ex} yield the algebraic
equation~\eqref{strip.ev} for the values of $\lambda$ for which $0\in \sigma(T_1(\la) + \sigma_j)$.

Since the eigenfunctions of the Dirichlet problem in $L^2((-\ell,\ell))$ form an orthonormal basis of this space, it can be shown that indeed
$\sigma(T) = \cup_{j \in \N} \, \sigma (T_1 + \sigma_j)$.

Finally, a (formal) inspection of equation~\eqref{strip.ev} shows that for fixed $k \in\N_0$ the eigenvalues with positive real parts are of the form
\begin{equation}\label{eigasympt}
\lambda_{jk}(\fra_0,\frq_0) = \frac{\ds \pi \ii}{\ds 2\ell} -\fra_{0} + \BigO\left(j^{-1/2}\right), \quad j \to +\infty,
\end{equation}
which proves the last claim.
\end{proof}

The eigenvalues of $G$ in Proposition~\ref{prop:T.strip},
computed from equation~\eqref{strip.ev}, are shown in Figure~\ref{fig.strip} for the case $\frq_0=0$ and $\fra_0=0$; there 
the sequences given by~\eqref{eigasympt} are clearly visible for each value of~$k\in\N_0$.
\begin{figure}[htb!]
\includegraphics[width=0.9 \textwidth]{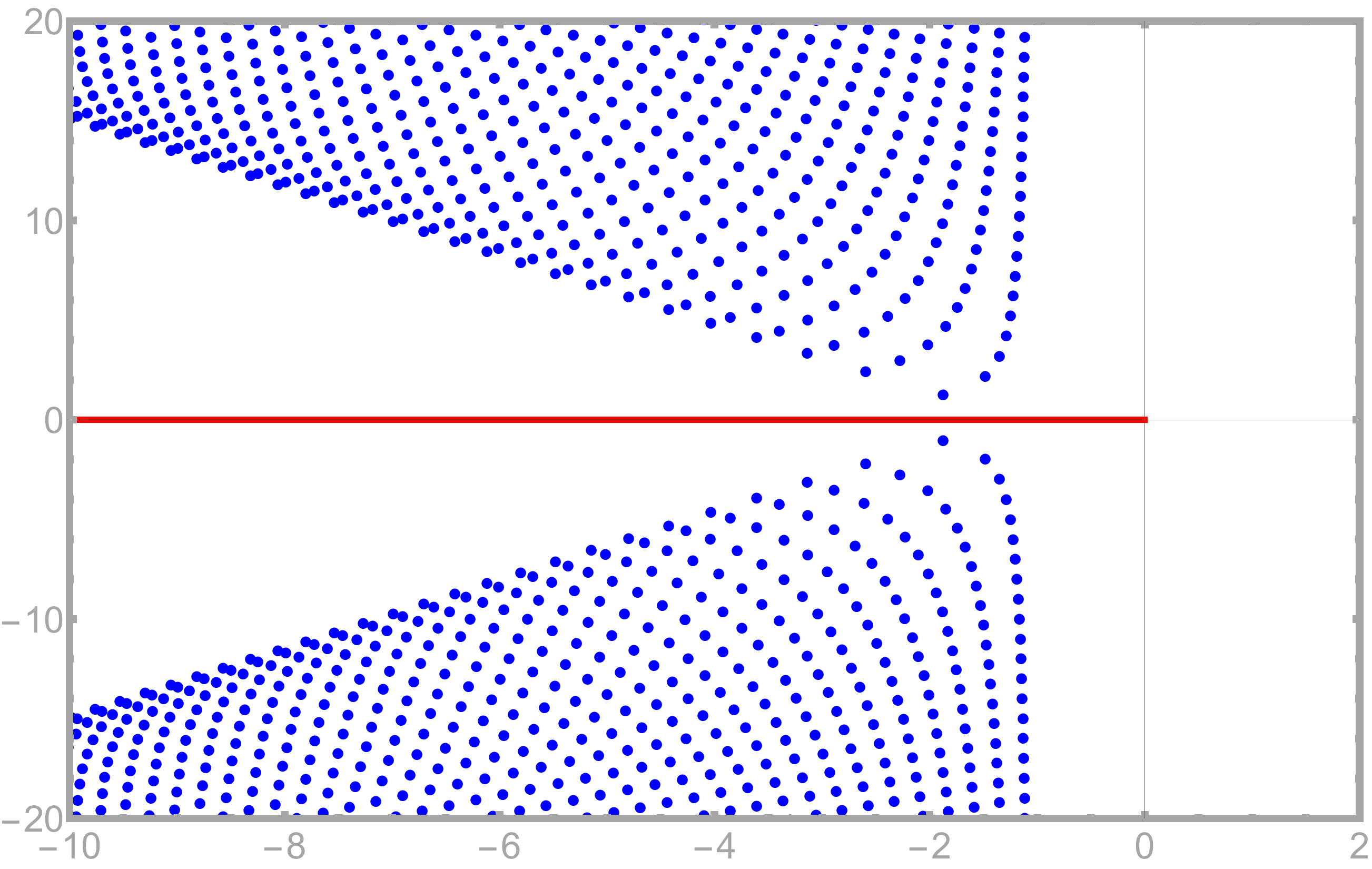}
\caption{Spectrum of $G$ with $a(x,y) = x^2$ and $q \equiv 0$ in $\Omega = \R\times(-\ell,\ell)$,  \cf~Proposition~\ref{prop:T.strip} with $\fra_0\!=\!1$, with eigenvalues $\{\lambda_{jk}{(1,0)}\}_{jk} \cup \{\ov{\lambda_{jk}{(1,0)}}\}_{jk}$ (dots/blue) and
the essential spectrum on the semiaxis $(-\infty,0]$ (bold/red).
}
\label{fig.strip}
\end{figure}
%


{\footnotesize
	\bibliographystyle{acm}
	\bibliography{references}

\begin{thebibliography}{10}

\bibitem{Adams-2003}
{\sc Adams, R.~A., and Fournier, J. J.~F.}
\newblock {\em {Sobolev spaces}}, 2nd~ed.
\newblock Elsevier, Amsterdam, 2003.

\bibitem{Boegli-2014-PhD}
{\sc B{\"o}gli, S.}
\newblock {\em {Spectral approximation for linear operators and applications}}.
\newblock PhD thesis, University of Bern, 2014.

\bibitem{Boegli-2017-42}
{\sc B{\"o}gli, S., Siegl, P., and Tretter, C.}
\newblock {A}pproximations of spectra of {S}chr{\"o}dinger operators with
  complex potential on {$\R^d$}.
\newblock {\em Comm. Partial Differential Equations 42\/} (2017), 1001--1041.

\bibitem{Davies-1989}
{\sc Davies, E.~B.}
\newblock {\em {Heat kernels and spectral theory}}.
\newblock Cambridge University Press, 1989.

\bibitem{Davies-1995}
{\sc Davies, E.~B.}
\newblock {\em {Spectral theory and differential operators}}.
\newblock Cambridge University Press, 1995.

\bibitem{Davies-2007}
{\sc Davies, E.~B.}
\newblock {\em {Linear operators and their spectra}}.
\newblock Cambridge University Press, 2007.

\bibitem{EE}
{\sc Edmunds, D.~E., and Evans, W.~D.}
\newblock {\em {Spectral Theory and Differential Operators}}.
\newblock Oxford University Press, New York, 1987.

\bibitem{Everitt-1978-79}
{\sc Everitt, W.~N., and Giertz, M.}
\newblock {Inequalities and separation for {S}chr\"odinger type operators in
  {$L_{2}({\bf R}^{n})$}}.
\newblock {\em Proc. Roy. Soc. Edinburgh Sect. A 79\/} (1978), 257--265.

\bibitem{Freitas-1999-78}
{\sc Freitas, P.}
\newblock {Spectral sequences for quadratic pencils and the inverse spectral
  problem for the damped wave equation}.
\newblock {\em J. Math. Pures Appl. 78}, 9 (1999), 965--980.

\bibitem{Freitas-1996-132}
{\sc Freitas, P., and Zuazua, E.}
\newblock {Stability results for the wave equation with indefinite damping}.
\newblock {\em J. Differential Equations 132\/} (1996), 338--352.

\bibitem{Ikehata-2017}
{\sc Ikehata, R., and Takeda, H.}
\newblock {Uniform energy decay for wave equations with unbounded damping
  coefficients}.
\newblock arXiv:1706.03942 [math.AP], 2017.

\bibitem{Jacob-2017}
{\sc Jacob, B., Tretter, C., and Trunk, Carsten und~Vogt, H.}
\newblock Systems with strong damping and their spectra.
\newblock Math. Methods Appl. Sci., 2017.
\newblock submitted.

\bibitem{Jacob-2009-79}
{\sc Jacob, B., and Trunk, C.}
\newblock {Spectrum and analyticity of semigroups arising in elasticity theory
  and hydromechanics.}
\newblock {\em Semigroup Forum 79\/} (2009), 79--100.

\bibitem{Kato-1952-125}
{\sc Kato, T.}
\newblock {Notes on some inequalities for linear operators}.
\newblock {\em Math. Ann. 125\/} (1952), 208--212.

\bibitem{Kato-1966}
{\sc Kato, T.}
\newblock {\em {Perturbation theory for linear operators}}.
\newblock Springer-Verlag, Berlin, 1995.

\bibitem{Krejcirik-2017-221}
{\sc Krej{\v{c}}i{\v{r}}{\'i}k, D., Raymond, N., Royer, J., and Siegl, P.}
\newblock {N}on-accretive {S}chr{\"o}dinger operators and exponential decay of
  their eigenfunctions.
\newblock {\em Israel J. Math. 221\/} (2017), 779--802.

\bibitem{Langer-2006-267}
{\sc Langer, H., Najman, B., and Tretter, C.}
\newblock {Spectral Theory of the Klein-Gordon Equation in Pontryagin Spaces}.
\newblock {\em Comm. Math. Phys. 267\/} (2006), 159--180.

\bibitem{Nakao-2001-238}
{\sc Nakao, M.}
\newblock {Energy decay for the linear and semilinear wave equations in
  exterior domains with some localized dissipations}.
\newblock {\em Math. Z. 238\/} (2001), 781--797.

\bibitem{Nashed-1976}
{\sc Nashed, M.~Z., and Votruba, G.~F.}
\newblock {\em A unified operator theory of generalized inverses}.
\newblock Academic Press, New York, 1976.

\bibitem{Reed4}
{\sc Reed, M., and Simon, B.}
\newblock {\em {Methods of Modern Mathematical Physics, Vol. 4: Analysis of
  Operators}}.
\newblock Academic Press, New York-London, 1978.

\bibitem{Sibuya-1975}
{\sc Sibuya, Y.}
\newblock {\em {Global theory of a second order linear ordinary differential
  equation with a polynomial coefficient}}.
\newblock North-Holland Publishing Co., Amsterdam, 1975.

\bibitem{Titchmarsh-1954-5}
{\sc Titchmarsh, E.~C.}
\newblock {On the asymptotic distribution of eigenvalues}.
\newblock {\em Q. J. Math. 5\/} (1954), 228--240.

\bibitem{Tretter-2008}
{\sc Tretter, C.}
\newblock {\em {Spectral Theory of Block Operator Matrices and Applications}}.
\newblock Imperial College Press, 2008.

\bibitem{Vainikko-1976}
{\sc Vainikko, G.}
\newblock {\em {Funktionalanalysis der Diskretisierungsmethoden}}.
\newblock B. G. Teubner Verlag, Leipzig, 1976.

\bibitem{Zuazua-1991-70}
{\sc Zuazua, E.}
\newblock {Exponential decay for the semilinear wave equation with localized
  damping in unbounded domains}.
\newblock {\em J. Math. Pures Appl. 70\/} (1991), 513--529.

\end{thebibliography}
}

\end{document}